\documentclass{article}

\usepackage{graphicx}%
\usepackage{multirow}%
\usepackage{amsmath,amssymb,amsfonts}%
\usepackage{amsthm}%
\usepackage{mathrsfs}%
\usepackage[title]{appendix}%
\usepackage{xcolor}%
\usepackage{textcomp}%
\usepackage{manyfoot}%
\usepackage{booktabs}%
\usepackage{algorithm}%
\usepackage{algorithmicx}%
\usepackage{algpseudocode}%
\usepackage{listings}%

\usepackage{hhline}
\usepackage{longtable}
\usepackage{easyReview}

\usepackage[margin=1.2in]{geometry}
\usepackage{mathtools}
\usepackage[affil-sl]{authblk}
\usepackage{hyperref}

\usepackage{cla_submission}

\theoremstyle{plain}%
\newtheorem{theorem}{Theorem}[section]
\newtheorem{proposition}[theorem]{Proposition}%
\newtheorem{lemma}[theorem]{Lemma}%
\newtheorem{cor}[theorem]{Corollary}%

\theoremstyle{definition}
\newtheorem{example}[theorem]{Example}%
\newtheorem{remark}[theorem]{Remark}%

\newtheorem{definition}[theorem]{Definition}%

\begin{document}

\title{A classification of nilpotent compatible Lie algebras} 

\author[1]{Manuel Ladra \thanks{manuel.ladra@usc.es}}
\author[1]{Bernardo Leite da Cunha \thanks{bernardo.mariz@rai.usc.es}}
\author[2]{Samuel A. Lopes \thanks{slopes@fc.up.pt}}
\affil[1]{Departamento de Matem\'aticas \& CITMAga, Universidade de Santiago de Compostela,
	15782, Spain} 
\affil[2]{CMUP, Departamento de Matem\'atica, Faculdade de Ci\^encias, Universidade do Porto, Rua do Campo Alegre s/n, 4169--007 Porto, Portugal.}

\maketitle

\abstract{Working over an arbitrary field of characteristic different from $2$, we extend the Skjelbred-Sund method to compatible Lie algebras and give a full classification of nilpotent compatible Lie algebras up to dimension $4$. In case the base field is cubically closed, we find that there are three isomorphism classes and a one-parameter family in dimension $3$, and $12$ isomorphism classes, $6$ one-parameter families and one $2$-parameter family in dimension $4$.
	
	MSC 17D99 Other nonassociative rings and algebras -- None of the above, but in this section 
	
	MSC 17B30 Lie algebras and Lie superalgebras -- Solvable, nilpotent (super)algebras 
	
	MSC 17A30 General nonassociative rings -- Nonassociative algebras satisfying other identities
}

\section{Introduction}

A compatible Lie algebra is a vector space $\g$ endowed with two Lie products $\sqb{-,-}$ and $\cb{-,-}$ such that any linear combination $\lambda \sqb{-,-}+\mu \cb{-,-}$ is still a Lie product, or equivalently, such that the following identity  holds for any $x, y, z \in \g$ (see Proposition~\ref{compatibleequiv})
\begin{align*}
	\begin{split}
		& \cb{\sqb{x,y},z}+\cb{\sqb{y,z},x}+\cb{\sqb{z,x},y} \\
		& + \sqb{\cb{x,y},z}+\sqb{\cb{y,z},x}+\sqb{\cb{z,x},y}=0.
	\end{split}
\end{align*}
Compatible Lie algebras are considered in many fields in mathematics and mathematical physics, such as the study of the classical Yang-Baxter equation~\cite{CompatibleYangBaxter}, integrable equations of the principal chiral model type~\cite{CompatibleChiral}, elliptic theta functions~\cite{CompatibleEliptic}, and loop algebras over Lie algebras~\cite{CompatibleLoop}. 

When, instead of considering general linear combinations of the form $\lambda \sqb{-,-}+\mu \cb{-,-}$, one fixes $\lambda=1$, the resulting product can be seen as an infinitesimal deformation of $\sqb{-,-}$. This follows by observing that the above identity is simply the cocycle identity for the adjoint module.  Thus, the study of compatible Lie algebras is also connected with deformation theory. In~\cite{MCCohomCLA}, the authors introduced a cohomology theory where they classify infinitesimal deformations of compatible Lie algebras via the second cohomology group. This cohomology theory is especially relevant to the present article.

In a broader sense, bi-Hamiltonian structures play an essential role  in the theory of integrable systems from mathematical physics. Such structures correspond to pairs of compatible Poisson brackets defined on the same manifold. In~\cite{CharacterFormulasBiHamiltonian06}, the authors have  studied the operads of compatible Lie  algebras and bi-Hamiltonian algebras.

This notion is a particular case of the idea of compatible algebraic structures. Two algebraic structures of the same type (i.e.~associative algebras, Lie algebras, Leibniz algebras, etc.) $(V, \circ)$ and $(V, \ast)$ with the same underlying vector space are said to be compatible if $(V, \lambda \cdot \circ + \mu  \cdot  \ast)$ has the same algebraic structure as the first ones for any scalars $\lambda, \mu$. Some non-Lie examples in the literature include compatible associative algebras in~\cite{OdesskiiSokolovPairsCompatibleAssoc2006} and~\cite{OdesskiiSokolovPairsCompatibleAssoc2008}, compatible associative bialgebras in~\cite{CompatibleAlgebrasMarquez}, compatible Lie bialgebras in~\cite{CompatibleLieBialgebras2015} and compatible Leibniz algebras in~\cite{CompatibleLeibniz2023} and~\cite{CohomDeformCompLeibninz2023}. General compatible structures have been studied from an operadic point of view in~\cite{OperadsCompatible08}.

As these algebraic structures are of interest, one of the first problems to consider is their classification. 
In this paper, we will restrict ourselves to the classification of nilpotent compatible Lie algebras.

Generalisations of the method developed by Skjelbred--Sund in~\cite{SkjelbredSund77} -- originally in regard to classifying nilpotent Lie algebras -- have been recently used to great effect for many kinds of extension problems.	
The original method and its generalisations allow one to classify central extensions of a given variety and consequently to obtain a classification of the nilpotent algebras up to a given dimension, as nilpotent algebras can always be obtained via central extensions of lower dimensional ones. Some of the varieties of algebras and types of extensions which have been classified using generalisations of the Skjelbred--Sund method are the following (we make no attempt at giving a complete list, see~\cite{KaygorodovHabilitation} and the references therein for a larger list of examples): associative algebras up to dimension $4$ in~\cite{NilpAssClass10}, nilpotent Lie algebras up to dimension $6$ in~\cite{NilpLieClass12}, nilpotent Jordan algebras up to dimension $5$ in~\cite{NilpJordanClass16}, central extensions of null-filiform and naturally graded filiform non-Lie Leibniz algebras in~\cite{NullFiliformLeib17}, nilpotent Malcev algebras up to dimension $6$ in~\cite{NilpMalcevClass18}, alternative, left alternative, Jordan, bicommutative, left commutative, assosymmetric, Novikov and left symmetric central extensions of null filiform associative algebras in~\cite{NullFilliformClass20}, nilpotent anticommutative algebras up to dimension $6$ in~\cite{NilpAlgClass21}, nilpotent algebras up to dimension $4$ in~\cite{NilpAnticommClass20} and nilpotent Poisson algebras up to dimension $4$ in~\cite{NilpPoissonClass}. This last reference is particularly interesting for us, as Poisson algebras are algebraic structures with \emph{two} products, as are compatible Lie algebras, and the specifics of the method we use are more closely related to this case.

This article is organised as follows. In Section~\ref{Prelims}, we give the basic definitions related to compatible Lie algebras, including the definition of cohomology as presented in~\cite{MCCohomCLA}. In Section~\ref{TheoreticalBases}, we start by defining the property of nilpotency in the compatible Lie algebra setting. We proceed by defining the notion of central extensions and relate them with cohomology, and then showing that any nilpotent compatible Lie algebra can be seen as a central extension of a lower dimensional one. Then we pose the problem of classifying extensions up to isomorphism and consider a generalisation of the Skjelbred--Sund method, proving the necessary results to show that it is a suitable tool for tackling the isomorphism question and thus for classifying nilpotent compatible Lie algebras. These preliminary sections come together in Section~\ref{Classification}, where we use the method to obtain a complete classification of nilpotent compatible Lie algebras up to dimension $4$. In case the base field is cubically closed, we find that there are three isomorphism classes and a one-parameter family in dimension $3$, and $12$ isomorphism classes, $6$ one-parameter families and one $2$-parameter family in dimension $4$.

\section{Preliminaries}\label{Prelims}

\subsection{Elementary definitions}

In this section we start by defining \emph{compatible Lie algebras} and the basics about them. We let $\KK$ be an arbitrary field of characteristic different from $2$.

\begin{proposition}\label{compatibleequiv}
	Let $\underline{\g}=(\g,\sqb{-,-})$ and $\undertilde{\g}=(\g,\cb{-,-})$ be two Lie algebras over the same vector space $\g$. Then the following conditions are equivalent:
	\begin{enumerate}[(i)]
		\item\label{comp1} $(\g, \dbl{-,-})$ is a Lie algebra, where $\dbl{x,y}=\sqb{x,y}+\cb{x,y}$ for all $x,y \in \g$;
		\item\label{comp2} $(\g, \dbl{-,-}_{\lambda,\lambda'})$ is a Lie algebra for all $\lambda, \lambda' \in \KK$, where \[\dbl{x,y}_{\lambda,\lambda'}=\lambda\sqb{x,y}+\lambda'\cb{x,y} \text{for all } x,y \in \g;\]
		\item\label{comp3} The following identity (named the \emph{mixed Jacobi identity}) holds for all $x,y,z \in \g$:
		\begin{align}\label{mixedjacobi}
			\begin{split}
				& \cb{\sqb{x,y},z}+\cb{\sqb{y,z},x}+\cb{\sqb{z,x},y} \\
				& + \sqb{\cb{x,y},z}+\sqb{\cb{y,z},x}+\sqb{\cb{z,x},y}=0.
			\end{split}
		\end{align}
	\end{enumerate}
\end{proposition}

With this we have the definition of compatible Lie algebras.

\begin{definition}
	A \emph{compatible Lie algebra} is a triple $(\g,\sqb{-,-},\cb{-,-})$, where $\underline{\g}=(\g,\sqb{-,-})$ and $\undertilde{\g}=(\g,\cb{-,-})$ are Lie algebras satisfying any of the three equivalent conditions in Proposition~\ref{compatibleequiv}.
\end{definition}

\begin{definition}
	A \emph{compatible Lie algebra homomorphism} between two compatible Lie algebras $(\g,\sqb{-,-}_\g,\cb{-,-}_\g)$ and $(\hh,\sqb{-,-}_\hh,\cb{-,-}_\hh)$ is a linear map $\varphi:\g \rightarrow \hh$ which is a Lie algebra homomorphism between $\underline{\g}$ and  $\underline{\hh}$, and  a Lie algebra homomorphism between $\undertilde{\g}$ and  $\undertilde{\hh}$.
\end{definition}

\begin{remark}
	This does not imply that $\varphi$ is a homomorphism between $\underline{\g}$ and $\undertilde{\hh}$ or vice-versa, which leads us to the next definition.
\end{remark}

\begin{definition}
	A \emph{skew-homomorphism} between two compatible Lie algebras $(\g,\sqb{-,-}_\g,\cb{-,-}_\g)$ and $(\hh,\sqb{-,-}_\hh,\cb{-,-}_\hh)$ is a linear map $\varphi:\g \rightarrow \hh$ which is a Lie algebra homomorphism between $\underline{\g}$ and  $\undertilde{\hh}$, and a Lie algebra homomorphism between $\undertilde{\g}$ and  $\underline{\hh}$.
\end{definition}

An invertible homomorphism is an \emph{isomorphism} and an invertible skew-homomorphism is a \emph{skew-isomorphism}. We write:
\begin{itemize}
	\item $\g\cong \hh$ if $\g$ and $\hh$ are \emph{isomorphic};
	\item $\g\skiso \hh$ if $\g$ and $\hh$ are \emph{skew-isomorphic};
	\item $\g\ambi \hh$ if $\g$ and $\hh$ are both isomorphic and skew-isomorphic, calling these \emph{ambi-isomorphic}.
\end{itemize}

\begin{definition}
	Let $(\g,\sqb{-,-},\cb{-,-})$ be a compatible Lie algebra. If we switch the order of the brackets, the resulting structure $(\g,\cb{-,-},\sqb{-,-})$ is still a compatible Lie algebra, which we denote by $\g^s$.
	It is clear that $\g\skiso \g^s$ via the identity map. We denote the fact that two algebras are switch copies of one another by $\hh=\g^s$.
\end{definition}

As expected, there is a notion of subalgebra and ideal.

\begin{definition}
	A \emph{subalgebra} of a compatible Lie algebra $\g$ is a vector subspace of $\g$ which is closed for both products.
	An \emph{ideal} $\mathfrak{i}$ of a compatible Lie algebra $\g$ is a vector subspace such that
	$
	\sqb{\mathfrak{i},\g}, \cb{\mathfrak{i},\g} \subseteq \mathfrak{i}. 
	$
\end{definition}

It holds true that the kernel of a homomorphism is an ideal of the domain and that the image of a homomorphism is a subalgebra of the codomain. The notion of quotient is well defined and the usual isomorphism theorems hold.

\begin{definition}
	The \emph{centre} of a compatible Lie algebra $\g$, denoted by $Z(\g)$, is the ideal defined by
	\[Z(\g)=\cb{x \in \g \mid \sqb{x,\g}=0=\cb{x,\g} } = Z(\underline{\g})\cap Z(\undertilde{\g}) .\]
\end{definition}

\begin{definition}
	A compatible Lie algebra is said to be \emph{abelian} if both its products are trivial, or equivalently, if it is equal to its centre.
\end{definition}

\begin{remark}\label{rmkautcentre}
	It is clear that an isomorphism $\varphi: \g \stackrel{\cong}{\longrightarrow} \hh$ maps $Z(\underline\g)$ to $Z(\underline\hh)$ and  $Z(\undertilde\g)$ to $Z(\undertilde\hh)$, and thus it also maps $Z(\g)$ to $Z(\hh)$. In the same way, a skew-isomorphism maps $Z(\underline\g)$ to $Z(\undertilde\hh)$ and  $Z(\undertilde\g)$ to $Z(\underline\hh)$, and therefore it also maps $Z(\g)$ to $Z(\hh)$.
\end{remark}

\subsubsection*{Some examples}

We provide some examples for gaining information. 

\begin{example}[A trivial example]
	Given a Lie algebra $(\g,\sqb{-,-})$, any multiple of its Lie product is still a Lie product. Thus, setting the second product $\cb{-,-}=\lambda\sqb{-,-}$, we obtain a compatible Lie algebra $(\g,\sqb{-,-},\cb{-,-})$. 
\end{example}

\begin{example}[A not-so-trivial example]
	Let $\g$ be a three-dimensional vector space generated by $x,y,z$. Define the following products:
	\begin{align*}
		\sqb{x,y}= z, & \qquad z \in Z(\underline\g), \text{ and} \\
		\cb{x,y}= z, & \qquad \cb{x,z}= 2x, \qquad \cb{y,z}= -2y.
	\end{align*}
	It is clear that the structures $\underline\g$ and $\undertilde\g$ are Lie algebras: $\underline\g$ is the Heisenberg Lie algebra of dimension $3$, and $\undertilde\g=\fsl_2 $.
	
	To check that this pair of products is compatible, we only need to check the mixed Jacobi identity on distinct basis elements. 
	We compute

	\begin{align*}
		\cb{\sqb{x,y},z}+\cb{\sqb{y,z},x}+\cb{\sqb{z,x},y} &= \cb{z,z}+\cb{0,x}+\cb{0,y} = 0; \\
		\sqb{\cb{x,y},z}+\sqb{\cb{y,z},x}+\sqb{\cb{z,x},y} &= \sqb{z,z}+\sqb{-2y,x}+\sqb{-2x,y} = 0+2z-2z =0;
	\end{align*}
	and thus, the mixed Jacobi identity, being the sum of the two expressions above, is equal to zero.
\end{example}

\begin{example}[A non-example]
	An obvious question is whether any two Lie products on the same vector space are compatible. This is false, as can be seen by the following non-example. Let $\g$ be a three-dimensional vector space generated by $x,y,z$. Define the following products:
	\begin{align*}
		\sqb{x,y}= x, & \qquad z \in Z(\underline\g), \text{ and} \\
		\cb{x,y}= z, & \qquad \cb{x,z}= 2x, \qquad \cb{y,z}= -2y.
	\end{align*}
	It is clear that the structures $\underline\g$ and $\undertilde\g$ are Lie algebras: $\underline\g = \g_2 \oplus \KK z$, where $\g_2$ is the non-abelian Lie algebra of dimension two, and $\undertilde\g=\fsl_2 $. But this pair of products is not compatible: 
	\begin{align*}
		\cb{\sqb{x,y},z}+\cb{\sqb{y,z},x}+\cb{\sqb{z,x},y} &= \cb{x,z}+\cb{0,x}+\cb{0,y} = 2y; \\
		\sqb{\cb{x,y},z}+\sqb{\cb{y,z},x}+\sqb{\cb{z,x},y} &= \sqb{z,z}+\sqb{-2y,x}+\sqb{-2x,y} \\ &= 2x-2x = 0.
	\end{align*}
	
\end{example}

\begin{example}
	Another question which may arise is whether two compatible Lie algebras are isomorphic if their component Lie algebras are isomorphic. This example shows that this is not true.
	
	Let $\hh$ be a three-dimensional vector space generated by $x,y,z$. Define the following two pairs of products:
	\begin{align*}
		\sqb{x,y}_1= z,  z \in Z(\underline\hh_1), &\text{ and} \cb{x,y}_1= z,  z \in Z(\undertilde\hh_1); \\
		\sqb{x,y}_2= z,   z \in Z(\underline\hh_2), &\text{ and} 	\cb{x,z}_2= y,   y \in Z(\undertilde\hh_2).
	\end{align*}
	It is an easy exercise to check that both pairs of products are compatible (as before, in dimension $3$ we only need to check the mixed Jacobi identity once, using all distinct basis vectors). We have that, as Lie algebras, $\underline\hh_1$, $\undertilde\hh_1$, $\underline\hh_2$ and $\undertilde\hh_2$ are all isomorphic to the three-dimensional Heisenberg Lie algebra. However, $Z(\hh_1)=\ab{z}$ and $Z(\hh_2)=0$, so $\hh_1 \not\cong \hh_2$.
	
\end{example}

\subsection{Cohomology of compatible Lie algebras}

In this subsection, we follow the definitions in~\cite{MCCohomCLA}, where the authors introduce and develop a cohomology theory with coefficients in a representation.

\begin{definition}[\cite{CompatibleLieBialgebras2015}]\label{representation}
	A \emph{representation} of a compatible Lie algebra $\g$ is a triple $(V,\rho,\mu)$, where $(V,\rho)$ is a representation of $\underline{\g}$, $(V,\mu)$ is a representation of $\undertilde\g$, and $(V,\rho+\mu)$ is a representation of $(\g,\sqb{-,-}+\cb{-,-})$.

\end{definition}

Let $(\g, \sqb{-,-})$ be a Lie algebra and let $(V,\rho)$ be a representation of $\g$.

Recall the Chevalley-Eilenberg complex (cf.~\cite{WeibelHomAlg}) given  by $C^n(\g,V)=\Hom_\KK(\bigwedge^n \g,V)$, for $n \geq 0$, and  differentials:
\begin{align*}
	(d^nf)(x_1,\dots,x_{n+1}) &= \sum_{i=1}^{n+1} (-1)^{i+1} \rho(x_i)f(x_1,\dots,\hat{x}_i,\dots,x_{n+1}) \\
	& +\sum_{1\leq i<j\leq n+1} (-1)^{i+j} f([x_i,x_j],x_1,\dots,\hat{x}_i,\dots,\hat{x}_j,\dots,x_{n+1}).
\end{align*}
We also define the spaces $\Cd^0(\g,V)=\cb{v \in V \mid \rho(x)v=\mu(x)v, \  x \in \g}$ and $\Cd^n(\g,V)=\displaystyle\bigoplus^n C^n(\g,V)$ (i.e.~the direct sum of $n$ copies of $C^n(\g,V)$).%

We define the following family of maps $\delta^n:\Cd^n(\g,V) \rightarrow \Cd^{n+1}(\g,V)$ by
\begin{gather*}\label{cCEcomplex}
	\delta^0v=d^0_{\rho}v=d^0_{\mu}v, \quad \delta^1(f) = (d^1_\rho(f),d^1_\mu(f)) \text{, and} \\
	\\
	\delta^n(\omega_1, \ldots, \omega_{n})=
	(d^n_{\rho}\omega_1,d^n_{\mu}\omega_1+d^n_{\rho}\omega_2,\ldots,\underbrace{d^n_{\mu}\omega_i+d^n_{\rho}\omega_{i+1}}_{(i+1)\text{-st position}},\ldots, d^n_{\mu}\omega_{n}),
\end{gather*}
where
$d^i_{\rho}$ (resp. $d^i_{\mu}$) is the $i$-th Chevalley-Eilenberg differential for the representation $(V,\rho)$ of $\underline{\g}$ (resp. for the representation $(V,\mu)$ of $\undertilde{\g}$).

\begin{definition}
	The \emph{compatible Lie algebra cohomology of $\g$ with coefficients in the representation $(V,\rho,\mu)$} is the cohomology associated with the above cochain complex.
\end{definition}

We now show the explicit expressions for the $2$-cocycles and $2$-coboundaries for the case where we consider a trivial representation of $\g$, as these are used in Section~\ref{TheoreticalBases}. 

The following result is a direct consequence of the definitions above.

\begin{proposition}\label{cocycleids}
	When considering the cohomology with trivial coefficients, the $2$-cocycles $\omega \in Z^2(\g,V)$ have the form $\omega=(\underline{\omega},\undertilde{\omega})$ where $\underline{\omega},\undertilde{\omega}: \g \times \g \rightarrow V$ are alternating bilinear maps which satisfy
	\begin{align*}
		&\underline{\omega}(\sqb{x,y},z)+ \underline{\omega}(\sqb{z,x},y) + \underline{\omega}(\sqb{y,z},x) =0, \\
		&\undertilde{\omega}(\cb{x,y},z)+ \undertilde{\omega}(\cb{z,x},y) + \undertilde{\omega}(\cb{y,z},x) =0, \text{ and } \\
		&\underline{\omega}(\cb{x,y},z)+ \underline{\omega}(\cb{z,x},y) + \underline{\omega}(\cb{y,z},x)  \\
		&+ \undertilde{\omega}(\sqb{x,y},z)+ \undertilde{\omega}(\sqb{z,x},y) + \undertilde{\omega}(\sqb{y,z},x) =0,
	\end{align*}
for all $x,y,z \in \g$.

Moreover, the $2$-coboundaries $\delta\varphi \in B^2(\g,V)$ have the form $\delta\varphi=(\underline{\delta\varphi},\undertilde{\delta\varphi})$ for a linear map $\varphi:\g \rightarrow V$, where
\[
\underline{\delta\varphi}(x,y)=\varphi(\sqb{x,y}), \qquad \undertilde{\delta\varphi}(x,y)=\varphi(\cb{x,y}).
\]
\end{proposition}

\section{Compatible structure theory}\label{TheoreticalBases}

In what follows, all algebras and vector spaces are assumed to be finite dimensional. 

\subsection{Nilpotent compatible Lie algebras}

\begin{definition}
	Let $\mathfrak{s}$ and $\mathfrak{t}$ be two subspaces of a compatible Lie algebra $\g$. We define the commutator $\dbl{\mathfrak{s},\mathfrak{t}}$ of $\mathfrak{s}$ and $\mathfrak{t}$ as
	\begin{align*}
		\dbl{\mathfrak{s},\mathfrak{t}} &= \sqb{\mathfrak{s},\mathfrak{t}} + \cb{\mathfrak{s},\mathfrak{t}} = \spann_\KK\cb{{\sqb{s,t}, \cb{s,t} \mid s \in \mathfrak{s}, t \in \mathfrak{t}}}.
	\end{align*}
\end{definition}

\begin{remark}
	Whenever $\mathfrak{i}$ and $\mathfrak{j}$ are ideals of $\g$, their commutator $\dbl{\mathfrak{i},\mathfrak{j}}$ is a subalgebra. Moreover, if we take an ideal $\mathfrak{i}$, then $\dbl{\g,\mathfrak{i}}$ is an ideal.
\end{remark}

Next, we define the lower and upper central series. We follow this definition with the usual conditions which characterise nilpotency.

\begin{definition}
	Given a compatible Lie algebra $\g$, we define the \emph{lower central series}
	\[
	\ZZZ_0(\g) \supseteq \ZZZ_1(\g) \supseteq \cdots \supseteq \ZZZ_n(\g) \supseteq \cdots
	\]
	by setting
	\[
	\ZZZ_0(\g):= \g, \qquad \ZZZ_i(\g):= \dbl{\g,\ZZZ_{i-1}(\g)},\text{ for } i > 0,
	\]
	and the \emph{upper central series}
	\[
	\ZZZ^0(\g) \subseteq \ZZZ^1(\g) \subseteq \cdots \subseteq \ZZZ^n(\g) \subseteq \cdots
	\]
	by setting	
	\[
	\ZZZ^0(\g)=\cb{0}, \qquad \ZZZ^i(\g)/\ZZZ^{i-1}(\g) = Z\nb{\g/\ZZZ^{i-1}(\g)},\text{ for } i > 0.
	\]
	
\end{definition}

\begin{lemma}\label{equivnilp}
	The following conditions are equivalent:
	\begin{enumerate}[(i)]
		\item\label{nilp1} $\ZZZ_n(\g)= \{0\}$ for sufficiently large $n$;
		\item\label{nilp2} $\ZZZ^n(\g)= \g$ for sufficiently large $n$.
		\item\label{nilp3} There is a decreasing sequence of ideals of $\g$
		\[\g_0 \supseteq \g_1 \supseteq \cdots \supseteq \g_n,\]
		where $\g_0=\g$, $\g_n = \{0\}$, and $\dbl{\g,\g_i}\subseteq \g_{i+1}$ for all $0 \leq i < n$.
		\item\label{nilp4} Any product $(x_0, (x_1, \cdots (x_{n-1},x_n))) $ involving $n+1$ elements is equal to zero, where each $(-,-)$ might be either $\sqb{-,-}$ or $\cb{-,-}$.
	\end{enumerate}
\end{lemma}

The proof of this result is entirely analogous to the classical case and is thus omitted.

\begin{definition}
	A compatible Lie algebra is said to be \emph{nilpotent} if it satisfies any of the equivalent conditions of Lemma~\ref{equivnilp}.
\end{definition}

A direct consequence of this definition is the following observation.

\begin{lemma}\label{smallernilp}
	Let $\g$ be a compatible Lie algebra. If $\g\neq 0$, then $\g$ 
	has non-trivial centre. Moreover, $\g$ is nilpotent if and only if $\g/Z(\g)$ is nilpotent.
\end{lemma}

It is clear that if $\g$ is a nilpotent compatible Lie algebra, then both $\underline{\g}$ and $\undertilde{\g}$ are nilpotent Lie algebras, as $\ZZZ_n(\underline\g),\ZZZ_n(\undertilde\g) \subseteq \ZZZ_n(\g)$.	The obvious question now is whether the converse holds.
A previous example provides a counterexample.

\begin{example}
	Let $\hh$ be a three-dimensional vector space generated by $x,y,z$. Define the following products:
	\begin{align*}
		\sqb{x,y}= z, \: z \in Z(\underline\hh), \text{ and} 
		\cb{x,z}= y, \: y \in Z(\undertilde\hh).
	\end{align*}
	Each of the component Lie algebras is a copy of the Heisenberg Lie algebra, which is nilpotent. But the centre of $\hh$ is zero, thus $\hh$ cannot be nilpotent.
\end{example}

\subsection{The classification problem}

Here, we present the definitions and results which will lead to a complete classification in low dimensions of the variety of nilpotent compatible Lie algebras. What follows is essentially a rephrasing of~\cite[Section 2]{NilpPoissonClass}, 
generalising the method used by Skjelbred and Sund in~\cite{SkjelbredSund77} to algebras with two products.

\subsubsection*{Central extensions and cocycles}	

From this point forward, the cohomology of compatible Lie algebras will always be assumed to have trivial coefficients.

\begin{definition}
	A \emph{central extension} of a compatible Lie algebra $\g$ by $V$ is a short exact sequence of compatible Lie algebras of the form
	\begin{equation*}
		\begin{tikzcd}
			0 \arrow[r] & V \arrow[r,"i"] & \ee \arrow[r,"p"] & \g \arrow[r] & 0,
		\end{tikzcd}
	\end{equation*}
	where $i(V) \subseteq Z(\ee)$.
	In particular, $V$ is an abelian compatible Lie algebra and $\ee/i(V) \cong \g$.
\end{definition}

\begin{remark}\hfil
	\begin{enumerate}
		\item A central extension is a particular case of an abelian extension(cf.~\cite[Definition 4.3.]{MCCohomCLA}) where the action is trivial, that is, with $\rho$ and $\mu$ are both zero. 
		\item As vector spaces, $\ee \cong \g \oplus V$. 
		We will interchangeably use the notation for external or internal direct sum. Thus, an element in $\g \oplus V$ can be written both as $(x,u)$ or $x+u$, with $x\in \g$ and $u \in V$.
	\end{enumerate}
\end{remark}

We may use $2$-cocycles to define central extensions (more generally, $2$-cocycles with non-trivial coefficients can be used to define abelian extensions, cf.~\cite[Section 4.2]{MCCohomCLA}). When we restrict to the nilpotent case, this will be the basis for our classification method.

\begin{definition}\label{cextdef}
	Given a compatible Lie algebra $\g$ and a cocycle $\omega \in Z^2(\g,V)$, we can construct a central extension of $\g$ by $V$ as follows. We set $\g_\omega = \g \oplus V$ and define the products
	\[
	\sqb{(x,u),(y,v)}_\omega = \nb{\sqb{x,y},\underline{\omega}(x,y)},\]
	\[\cb{(x,u),(y,v)}_\omega = \nb{\cb{x,y},\undertilde{\omega}(x,y)},
	\]
	where $x,y \in \g$ and $u,v \in V$. We further take $i:V \rightarrow \g_\omega$ to be the inclusion and $p:\g_\omega \rightarrow \g$ to be the projection, with respect to the decomposition $\g_\omega=\g\oplus V$.
\end{definition}

From the definition, it is clear that indeed $\g_\omega$ is a compatible Lie algebra: the Jacobi identities result from the cocycle identities (see Proposition~\ref{cocycleids}). Moreover, it is also trivial to check that $V \subseteq Z(\g_\omega)$. We can refine this by introducing the notion of the annihilator of a cocycle.
\medskip

\begin{definition}\label{extensioncentre}
	For a cocycle $\omega=(\underline\omega,\undertilde\omega) \in Z^2(\g,V)$, we define $\ann(\omega) = \ann(\underline\omega) \cap \ann(\undertilde\omega)$, where, for an alternating bilinear map $\beta:\g \times \g \rightarrow V$, $\ann(\beta)=\cb{x \in \g \mid \beta(x,y)=0, \: \forall y \in \g}$. Then, $Z(\g_\omega)=(Z(\g)\cap \ann(\omega))\oplus V$. We will use this fact throughout.
\end{definition}

\begin{definition}
	If a central extension $\ee$ of $\g$ by $V$ is such that $i(V)=Z(\ee)$, we say it is a \emph{full} central extension.
\end{definition}
Thus, the central extension $\g_\omega$ is full if and only if $Z(\g)\cap \ann(\omega)=0$. If this holds, we say the cocycle $\omega$ is \emph{admissible}.

The next proposition shows that any nonzero nilpotent compatible Lie algebra can be seen as a full central extension of a smaller one.

\begin{proposition}\label{nilpext1}
	Let $\g$ be a compatible Lie algebra. Then there always exists a compatible Lie algebra $\hh$ such that $\g$ is a full central extension of $\hh$, and moreover, $\hh$ is unique up to isomorphism. 
\end{proposition}

\begin{proof}
	The following is a full central extension 
	\begin{equation*}
		\begin{tikzcd}
			0 \arrow[r] & Z(\g) \arrow[r,"i"] & \g \arrow[r,"p"] & \hh:=\g/Z(\g) \arrow[r] & 0.
		\end{tikzcd}
	\end{equation*}
	Moreover, given another full central extension
	\begin{equation*}
		\begin{tikzcd}
			0 \arrow[r] & Z(\g) \arrow[r,"i'"] & \g \arrow[r,"p'"] & \hh' \arrow[r] & 0,
		\end{tikzcd}
	\end{equation*}
	exactness ensures that $\hh'\cong \g/Z(\g)=\hh$.
\end{proof}

In particular, if $\g$ is nonzero, finite-dimensional, and nilpotent, then by Lemma~\ref{smallernilp} $Z(\g)\neq 0$, so $\hh = \g/Z(\g)$ is nilpotent and $\dim \hh < \dim \g$.

\begin{theorem}\label{nilpext}
	Let $\g$ be a nilpotent compatible Lie algebra and $\hh=\nolinebreak\g/Z(\g)$. %
	Then the extension of Proposition~\ref{nilpext1} can be seen as an extension $\hh_\omega$ by a cocycle $\omega\in Z^2(\hh,Z(\g))$.	
\end{theorem}

\begin{proof}
	In this proof, unadorned products will always refer to $\g$.
	
	Consider the extension $\begin{tikzcd}
		0 \arrow[r] & Z(\g) \arrow[r,"i"] & \g \arrow[r,"p"] & \hh \arrow[r] & 0,
	\end{tikzcd}$ where $i$ is the inclusion and $p$ the canonical epimorphism.
	
	Fix a vector space decomposition $\g=\hh'\oplus Z(\g)$ such that $p_{\mid \hh'}: \hh' \rightarrow \hh$ is a vector space isomorphism. Let $j:={p_{\mid \hh'}}^{-1}$ and define a cocycle $\omega=(\underline{\omega},\undertilde{\omega})\in Z^2(\hh,Z(\g))$ by
	\begin{align*}
		\underline{\omega}(x,y)&=\sqb{j(x),j(y)}-j(\sqb{x,y}_\hh), \\
		\undertilde{\omega}(x,y)&=\cb{j(x),j(y)}-j(\cb{x,y}_\hh),
	\end{align*}
	where $x,y \in \hh$. These maps have values in $\g$, but note that
	\[
	p(\underline{\omega}(x,y))=\sqb{p\circ j(x),p\circ j(y)}_\hh-p\circ j(\sqb{x,y}_\hh)=\sqb{x,y}_\hh-\sqb{x,y}_\hh=0,
	\]
	and thus $\underline{\omega}(x,y)\in Z(\g)$; similarly, $\undertilde{\omega}(x,y)\in Z(\g)$.
	
	The fact that $\omega=(\underline{\omega},\undertilde{\omega})$ is alternating and satisfies the cocycles identities follows from the Jacobi identities in $\g$ and $\hh$ and the linearity of $j$.
	
	We now construct $\hh_\omega$. 
	For $x,y \in \hh$ and $u,v \in Z(\g)$, we have
	\begin{align*}
		\sqb{(x,u),(y,v)}_\omega &= \sqb{x,y}_{\hh}+\underline{\omega}(x,y)= \sqb{x,y}_{\hh} +\sqb{j(x),j(y)}-j(\sqb{x,y}_\hh),
	\end{align*}
	and similarly for $\cb{-,-}_\omega$.
	
	Consider the vector space isomorphism 
	\begin{align*}
		\Phi := j \oplus \id_{Z(\g)}  \colon \hh_\omega = \hh\oplus Z(\g) & \longrightarrow \g = \hh'\oplus Z(\g)\\
		(x,u) & \longmapsto j(x)+u.
	\end{align*}
	Then, 
	\begin{align*}
		\Phi\nb{\sqb{(x,u),(y,v)}_\omega} &= \Phi\big(\underbrace{\sqb{x,y}_\hh}_{\in \hh}+\underbrace{\sqb{j(x),j(y)}-j(\sqb{x,y}_\hh)}_{\in Z(\g)}\big)\\
		&=j(\sqb{x,y}_\hh)+\sqb{j(x),j(y)}-j(\sqb{x,y}_\hh)= \sqb{j(x),j(y)} \\
		&=\sqb{j(x)+u,j(y)+v}= \sqb{\Phi((x,u)),\Phi((y,v))}.
	\end{align*}
	And similarly for $\cb{-,-}$, so $\Phi$ is an isomorphism of compatible Lie algebras.
\end{proof}

\begin{remark}
	In the setting of Theorem~\ref{nilpext}, \[Z(\g)=Z(\hh_\omega)=\nb{Z(\hh)\cap \ann \omega}\oplus Z(\g), \] 
	so $Z(\hh)\cap\ann(\omega)=0$.
\end{remark}

The uniqueness of $\hh$ in Proposition~\ref{nilpext1} can be rephrased as follows.

\begin{theorem}\label{nilpext2}
	Let $\g_\omega$ and $\hh_\theta$ be isomorphic full central extensions of nilpotent compatible Lie algebras $\g$ and $\hh$. Then, $\g\cong\hh$. In particular, the spaces $Z^2(\g,Z(\g_\omega))$ and $Z^2(\hh,Z(\hh_\theta))$ are isomorphic and can be identified.
\end{theorem}

This sets us up perfectly to attempt a classification problem. Combining this result with Lemma~\ref{smallernilp}, we have the following corollaries.

\begin{cor}\label{corsmaller}
	Any nonzero finite-dimensional nilpotent compatible Lie algebra can be seen as a full central extension by cocycles of a nilpotent compatible Lie algebra of strictly smaller dimension.
\end{cor}

Iterating this, we obtain the following.

\begin{cor}
	Any finite-dimensional nilpotent compatible Lie algebra can be obtained via a series of full central extensions by cocycles starting with the compatible Lie algebra of dimension zero. 
\end{cor}

For a classification result, we need to be able to determine whether two extensions of the same algebra are isomorphic. 
Once we have a suitable criterion or method, we are able to classify nilpotent compatible Lie algebras up to a target dimension by starting with the compatible Lie algebra of dimension zero and extending successively via full central extensions.

This motivates the following results.

\subsubsection*{Cohomology, actions and orbits}


\begin{lemma}\label{cohomiso}
	Let $\g$ be a compatible Lie algebra and $V$ a vector space. Given two cocycles $\omega, \theta \in Z^2(\g,V)$ such that $\sqb{\omega}=\sqb{\theta}$, we have $\g_\omega \cong \g_\theta$.
\end{lemma}

\begin{proof}
	This is immediate from the fact that $H^2(\g,V)$ classifies equivalence classes of central extensions (see~\cite[Theorem 4.4]{MCCohomCLA}).
\end{proof}

In spite of Lemma~\ref{cohomiso}, two non-equivalent central extensions may yield isomorphic compatible Lie algebras. Our next aim is to find a partial converse of Lemma~\ref{cohomiso}. 

Simple direct computations prove the following result.

\begin{proposition}
	Let $\g$ be a compatible Lie algebra, $V$ a vector space, and let $\Aut(\g)$ denote the group of automorphisms of $\g$ and $\GL(V)$ denote the group of automorphisms of $V$. Then there is a right action of $\Aut(\g)$ on $Z^2(\g,V)$ and a left action of $\GL(V)$ on $Z^2(\g,V)$ given by
	\begin{align}
		(\omega \phi)(x,y)&=\omega\nb{\phi(x),\phi(y)},\label{actaut} \\
		f \omega &= (f \circ \underline{\omega}, f \circ \undertilde{\omega}),\label{actgl}
	\end{align}
	respectively, where $\omega \in Z^2(\g,V)$,  $\phi \in \Aut(\g)$, $f \in \GL(V)$  and $x,y \in \g$. Moreover, these actions satisfy $f(\omega \phi)=(f \omega)\phi$ and are induced to $H^2(\g,V)$.
\end{proposition}

\begin{remark}\!
	\begin{enumerate}[(i)]
		\item Fixing a basis for $\g$, the action of $\phi$ on $\omega$ is given by a change of basis induced by $\phi$. Thus, we may also compute $\omega\phi$ by doing the matrix computation $\phi^t\omega\phi$;
		\item The condition that a cocycle $\omega$ is admissible depends only on the $\Aut(\g)$-orbit of $\omega$, since automorphisms of $\g$ preserve $Z(\g)$. 
	\end{enumerate}
\end{remark}

These actions enable us to establish a criterion to ensure an isomorphism between extensions.

\begin{theorem}\label{maintheorem1}
	Let $\g$ be a compatible Lie algebra and $V$ a vector space. Let $\omega$ and $\theta$ be two cocycles of $Z^2(\g,V)$ such that the corresponding central extensions are full (or equivalently, $Z(\g_\omega)=Z(\g_\theta)=V$). Under these conditions, $\g_\omega\cong \g_\theta$ if and only if there exist $\phi \in \Aut(\g)$ and $f \in \GL(V)$ such that $ \sqb{\theta\phi} =\sqb{f\omega} $.

\end{theorem}

\begin{proof}
	We start by proving the direct implication. 
	
	Suppose then that $\Phi : \g_\omega \rightarrow \g_\theta$ is an isomorphism. Since $Z(\g_\omega)=Z(\g_\theta)=V$, we can consider the restriction $f:=\Phi_{\mid V}\in \GL(V)$. We  also consider the maps $\phi:\g \rightarrow \g$ and $\varphi:\g \rightarrow V$ defined as follows.
	
	Let $i:\g \rightarrow \g_\omega$ be the canonical injection and let $p_\g : \g_\theta \rightarrow \g$ and $p_V : \g_\theta \rightarrow V$ be the canonical projections. We set $\phi := p_\g \circ \Phi \circ i$ and $\varphi := p_V \circ \Phi \circ i$.
	In other words, we have just defined $f$, $\phi$ and $\varphi$ so that, given $(x,u)\in \g_\omega$ with $x \in \g$ and $u \in V$, \[\Phi((x,u))=\big(\phi(x),\varphi(x)+f(u)\big).\]
	For $x,y \in \g$, we have
	\begin{align*}
		\Phi\nb{\sqb{x,y}_\omega} &= \Phi\nb{\nb{\sqb{x,y},\underline\omega(x,y)}} \\
		&= \nb{\phi\nb{\sqb{x,y}}, \varphi\nb{\sqb{x,y}}+f\nb{\underline\omega(x,y)}}, 		\text{ and } \\
		\\
		\sqb{\Phi(x),\Phi(y)}_\theta &= \sqb{(\phi(x),\varphi(x)),(\phi(y),\varphi(y))}_\theta \\ 
		&= \nb{\sqb{\phi(x),\phi(y)},\underline\theta(\phi(x),\phi(y))}.
	\end{align*}
	Since $\Phi$ is an isomorphism, we have that \[\underline\theta(\phi(x),\phi(y))=\varphi\nb{\sqb{x,y}}+f\nb{\underline\omega(x,y)}, \text{ and } \phi(\sqb{x,y})=\sqb{\phi(x),\phi(y)}, \] and in the same way, it is shown that
	\[\undertilde\theta(\phi(x),\phi(y))=\varphi\nb{\cb{x,y}}+f\nb{\undertilde\omega(x,y)}, \text{ and } \phi(\cb{x,y})=\cb{\phi(x),\phi(y)}.\]
	Thus, $\phi \in \Aut(\g)$ and, according to all the above definitions, we have $\theta\phi=\delta\varphi+f\omega $, which is the same as $ \sqb{\theta\phi} =\sqb{f\omega} $, as we wanted to show.

	Conversely, assume that there exist $\phi \in \Aut(\g)$ and $f \in \GL(V)$ such that $ \sqb{\theta\phi} =\sqb{f\omega} $, so write $\theta\phi=\delta\varphi+f\omega $. We construct a map
	$\Phi:\g_\omega \rightarrow \g_\theta$ by setting $\Phi((x,u)):=\nb{\phi(x),\varphi(x)+f(u)}$, for $x\in\g$ and $u\in V$. Reversing the above argument we see that $\Phi$ is an isomorphism.
\end{proof}

In theory, this result already allows us to develop a classification method. As we have already discussed, for classifying the nilpotent compatible Lie algebras of dimension $n$ (assuming we have the classification up to dimension $n-1$), we would simply consider the possible dimensions of the centre $1, 2, \ldots, n-1$. For each central dimension $s$ we obtain any $n$-dimensional algebra by extending one of the $(n-s)$-dimensional algebras already on our list, and with the above criterion we could prevent duplicate classes of isomorphism.

In practice, it might still be quite complex to understand the actions of $\Aut(\g)$ and $\GL(V)$ on the cohomology group $H^2(\g,V)$. With this in mind, we restate the above criterion in a more computational form.

\begin{lemma}\label{components}
	Let $\g$ be a compatible Lie algebra and $V$ a vector space, and fix a basis $\cb{e_1,\ldots,e_s}$ of $V$. Let $\omega=(\underline\omega,\undertilde\omega)\in Z^2(\g,V)$. Then $\omega$ can be uniquely written as
	\begin{align*}
		\underline\omega(x,y)=\sum_{i=1}^s \underline\omega_i(x,y)e_i, \qquad \undertilde\omega(x,y)=\sum_{i=1}^s \undertilde\omega_i(x,y)e_i,
	\end{align*}
	where $\omega_i=(\underline\omega_i,\undertilde\omega_i)\in Z^2(\g,\KK)$ for each $i=1,\ldots,s$. Moreover, $\omega\in B^2(\g,V)$ if and only if $\omega_i\in B^2(\g,\KK)$ for each $i=1,\ldots,s$.
\end{lemma}

From now on, $\cb{e_1,\ldots,e_s}$ is a fixed basis of $V$. 

\begin{proof}
	The fact we can write $\underline\omega(x,y)$ and $\undertilde\omega(x,y)$ uniquely as stated is immediate, so we need only prove that the component functions are in fact cocycles with values in $\KK$.
	We illustrate this with one of the identities:
	\begin{align*}
		&\underline{\omega}(\sqb{x,y},z)+ \underline{\omega}(\sqb{z,x},y) + \underline{\omega}(\sqb{y,z},x) \\
		= & \sum_{i=1}^s \underline\omega_i(\sqb{x,y},z)e_i + \sum_{i=1}^s \underline\omega_i(\sqb{z,x},y)e_i + \sum_{i=1}^s \underline\omega_i(\sqb{y,z},x)e_i \\
		=& \sum_{i=1}^s \bigg(\underline\omega_i(\sqb{x,y},z)+\underline\omega_i(\sqb{z,x},y)+ \underline\omega_i(\sqb{y,z},x)\bigg)  e_i;
	\end{align*}
	
	Thus we have $\underline{\omega}(\sqb{x,y},z)+ \underline{\omega}(\sqb{z,x},y) + \underline{\omega}(\sqb{y,z},x)=0$ if and only if \[\underline\omega_i(\sqb{x,y},z)+\underline\omega_i(\sqb{z,x},y)+ \underline\omega_i(\sqb{y,z},x)=0 \text{ for all } i.\]
	
	The second part follows from a similar argument. If $\varphi$ is a linear map from $\g$ to $V$, we can write
	\begin{align*}
		\underline{\delta\varphi}\nb{x,y} &= \varphi(\sqb{x,y})
		=\sum_{i=1}^s \varphi_i(\sqb{x,y}) e_i  
		= \sum_{i=1}^s \underline{\delta\varphi_i}\nb{x,y} e_i,
	\end{align*}
	where the $\varphi_i$ are the component functions. This works the same for $\undertilde{\delta\varphi}$ and these expressions show that a cocycle is of the form $\delta\varphi$ if and only if each component is of the form $\delta\varphi_i$, thus proving the second assertion.
\end{proof}

To make use of this lemma, we will need to impose some conditions on the algebras we are working with. This will lead to the following notion of algebras with central components, and the final result will
make it so that we will skip said algebras when we use the construction/classification method. This will not be an issue, as we will see that the algebras we skip are easy to construct and will not pose problems when it comes to identifying isomorphisms.

\begin{definition}\label{centralcomponent}
	Let $\g$ be a compatible Lie algebra. A central component of $\g$ is a one-dimensional ideal $\KK x$ where $x\in Z(\g)$ such that there is a subalgebra $\mathfrak{i}$ satisfying $\g = \mathfrak{i}\oplus \KK x$. 
\end{definition}


\begin{remark}\!
	\begin{enumerate}[(i)]
		\item 	The subalgebra $\mathfrak{i}$ in the above definition satisfies $\dbl{\g,\g}\subseteq \mathfrak{i}$, thus it is an ideal.
		\item We might consider an alternative notion of a central component by defining it as a nonzero central ideal $\mathfrak{z}$ such that $\g = \mathfrak{i}\oplus\mathfrak{z}$, for some subalgebra $\mathfrak{i}$ of $\g$. This is equivalent to the previous definition.
		\item The definition implies that, in the context of our construction method, the $n$-dimensional nilpotent algebras with central components are constructed by taking the direct sum of each $(n-1)$-dimensional nilpotent algebra with a one-dimensional central ideal. This justifies our previous claim that these algebras may be skipped without posing problems, as they are not pairwise isomorphic if we start with a list of $(n-1)$-dimensional nilpotent algebras having no isomorphic pairs.
	\end{enumerate}
	
\end{remark}

	%
%
%

\begin{lemma}\label{lindep}
	Let $\g$ be a compatible Lie algebra and $V$ be a vector space. Let $\omega \in Z^2(\g,V)$ such that the corresponding central extension is full. 
	Then, $\g_\omega$ has a central component if and only if the cohomology classes of the components $\omega_i$ are linearly dependent in $H^2(\g,\KK)$, or in other words, the rank of $\spann\cb{\sqb{\omega_1}, \sqb{\omega_2}, \ldots, \sqb{\omega_s}}\subset H^2(\g,\KK)$ is less than $s$. 
\end{lemma}

\begin{proof}
	Recall that $\ann(\omega)\cap Z(\g)=0$ is equivalent to $Z(\g_\omega)=V$. For the direct implication, assume $\g_\omega$ has a central component $\KK v_1 \subseteq V$, and extend the element $v_1$ to a basis $\cb{v_1,v_2,\ldots,v_s}$ of $V$. There exists a change of basis matrix $A=(a_{ij})$ from the basis $\cb{e_1,e_2,\ldots,e_s}$ to the basis $\cb{v_1,v_2,\ldots,v_s}$, that is,
	\[
	e_i = \sum_{j=1}^s a_{ij}v_j.
	\]
	We can rewrite the expressions in Lemma~\ref{components} as 
	\begin{align*}
		\underline\omega(x,y)=\sum_{i,j=1}^s a_{ij}\underline\omega_i(x,y)v_j, \qquad \undertilde\omega(x,y)=\sum_{i,j=1}^s a_{ij} \undertilde\omega_i(x,y)v_j.
	\end{align*}
	We note that, since $v_1 \not\in \dbl{\g_\omega,\g_\omega}$ and because $\omega(x,y)$ expresses the $V$ component of the products of $x$ and $y$, the $v_1$ component of $\omega(x,y)$ is equal to $(0,0)$. We can write that precisely as
	\begin{align*}
		\sum_{i=1}^s a_{i1}\underline\omega_i(x,y)=0, \qquad \sum_{i=1}^s a_{i1} \undertilde\omega_i(x,y) = 0,
	\end{align*}
	for all $x,y \in \g$, and this in turn can be rewritten as $\sum_{i=1}^s a_{i1}\omega_i = 0$. This of course implies that  $\sum_{i=1}^s a_{i1}\sqb{\omega_i} = 0$ and so the cohomology classes are linearly dependent, as the first column of $A$ is nonzero, 
	
	Conversely, assume that the set $\cb{\sqb{\omega_1}, \sqb{\omega_2}, \ldots, \sqb{\omega_s}}$ is linearly dependent. Reordering the cocycles, we can assure that
	\[
	\sqb{\omega_s} = \sum_{i=1}^{s-1} \alpha_i \sqb{\omega_i},
	\]
	For scalars $\alpha_i\in \KK$. Define a new cocycle $\omega'$ such that
	\begin{align*}
		\underline\omega'(x,y)=\sum_{i=1}^s \underline\omega'_i(x,y)e_i, \qquad \undertilde\omega'(x,y)=\sum_{i=1}^s \undertilde\omega'_i(x,y)e_i,
	\end{align*}
	where $\underline\omega'_i = \underline\omega_i$ and $\undertilde\omega'_i=\undertilde\omega_i$ for $i=1, \ldots, s-1$, and $\underline\omega'_s = \sum_{i=1}^{s-1} \alpha_i\underline\omega_i$ and $\undertilde\omega'_s= \sum_{i=1}^{s-1} \alpha_i\undertilde\omega_i$. It is clear that $\sqb{\omega'}=\sqb{\omega}$ and by Lemma~\ref{cohomiso}, $\g_\omega \cong \g_{\omega'}$. We can write 
	\begin{align*}
		\underline\omega'(x,y)=\sum_{i=1}^{s-1} \underline\omega_i(x,y)w_i, \qquad \undertilde\omega(x,y)=\sum_{i=1}^{s-1} \undertilde\omega_i(x,y)w_i,
	\end{align*}
	where, for $i=1, \ldots, s-1$, $w_i:=e_i+\alpha_ie_s$. 
	
	%
	We can complete the set $\cb{w_1, \ldots, w_{s-1}}$ to a basis  $\cb{w_1, \ldots, w_{s-1}, w_s}$ and  $\KK w_s$ is a central component of $\g_{\omega'}$, therefore $\g_\omega$ also has a central component.
\end{proof}

We can rephrase Theorem~\ref{maintheorem1} in terms of the coordinate cocycles (with the caveat that we have to exclude the algebras with central components).

\begin{theorem}\label{maintheorem2}
	Let $\g$ be a compatible Lie algebra and $V$ be a vector space. Let $\omega,\theta \in Z^2(\g,V)$. Suppose that $\g_\omega$ has no central components and $\ann(\omega)\cap Z(\g)=\ann(\theta)\cap Z(\g)=0$. Under these conditions, $\g_\omega \cong \g_\theta$ if and only if there exists $\phi \in \Aut(\g)$ such that 
	\[\spann\cb{{\sqb{\omega_i}}\mid 1 \leq s \leq n}=\spann\cb{\sqb{\theta_i\phi}\mid 1 \leq s \leq n}\subseteq H^2(\g,\KK),\]
	where the components $\omega_i$ and $\theta_i$ are the ones introduced in Lemma~\ref{components}.	
\end{theorem}

\begin{proof}
	Suppose first that $\g_\omega \cong \g_\theta$. By Theorem~\ref{maintheorem1}, there exist $\phi \in \Aut(\g)$ and $f \in \GL(V)$ such that $ \sqb{\theta\phi} =\sqb{f\omega} $. Write $f(e_i)=\sum_{j=1}^s a_{ij}e_j$. We compute the following
	\begin{align}
		(\underline\theta\phi-f\underline\omega)(x,y) &= \sum_{j=1}^s (\underline\theta_j\phi)(x,y)e_j-f\nb{\sum_{i=1}^s \underline\omega_i(x,y)e_i} \label{psiphiexpression}	 \\ \nonumber
		&= \sum_{j=1}^s (\underline\theta_j\phi)(x,y)e_j-\sum_{i=1}^s \sum_{j=1}^s a_{ij}\underline\omega_i(x,y)e_j	 \\ \nonumber
		&= \sum_{j=1}^s \nb{\underline\theta_j\phi-\sum_{i=1}^s a_{ij}\underline\omega_i}(x,y)e_j,
	\end{align}
	and the computation for the pair $\undertilde\omega$ and $\undertilde\theta$ is analogous. So, in the end, we can write
	\begin{equation}\label{psiphiexpression2}
		(\theta\phi-f\omega)(x,y)=\sum_{j=1}^s \nb{\theta_j\phi-\sum_{i=1}^s a_{ij}\omega_i}(x,y)e_j.
	\end{equation}
	By Lemma~\ref{components}, since $\theta\phi-f\omega \in B^2(\g,V)$, $\theta_j\phi-\sum_{i=1}^s a_{ij}\omega_i \in B^2(\g,\KK)$ for each $j$. This can be rewritten as
	\[
	\sqb{\theta_j\phi} = \sum_{i=1}^s a_{ij}\sqb{\omega_i} \in H^2(\g,\KK) \text{, for each } j = 1,\ldots, s,
	\]
	and thus $\spann\cb{{\sqb{\omega_i}}\mid 1 \leq s \leq n}=\spann\cb{\sqb{\theta_i\phi}\mid 1 \leq s \leq n}$, as we wanted to show.
	
	Conversely, now suppose that 
	\[\spann\cb{{\sqb{\omega_i}}\mid 1 \leq s \leq n}=\spann\cb{\sqb{\theta_i\phi}\mid 1 \leq s \leq n}\subseteq H^2(\g,\KK).\]
	Since $\g_\omega$ has no central components, the $\sqb{\omega_i}$ are linearly independent by Lemma~\ref{lindep}, and there is an invertible matrix $(a_{ij})$ such that
	\[
	\sqb{\theta_i\phi} = \sum_{j=1}^s a_{ij}\sqb{\omega_j} \text{, for } i = 1, \ldots, s
	\]
	or equivalently,
	\[
	\theta_i\phi - \sum_{j=1}^s a_{ij} \omega_j \in B^2(\g,\KK) \text{, for } i = 1, \ldots, s.
	\]
	We now define a map $f\in \GL(V)$ by setting $f(e_i):=\sum_{j=1}^s a_{ij}e_j$, and just as in~\eqref{psiphiexpression} and in~\eqref{psiphiexpression2} we compute 
	\begin{align*}
		(\theta\phi-f\omega)(x,y)=\sum_{j=1}^s \nb{(\theta_i\phi)e_j-\sum_{i=1}^s a_{ij}\omega_i}(x,y)e_j.
	\end{align*}
	Again by Lemma~\ref{components} but arguing in the opposite direction, we conclude that $ \sqb{\theta\phi} =\sqb{f\omega} $. By Theorem~\ref{maintheorem1}, we conclude that  $\g_\omega \cong \g_\theta$.
\end{proof}	
Now we are finally ready to introduce the remaining concepts and results in order to get a characterisation of the isomorphism classes of extensions in terms of orbits of the $\Aut(\g)$ action.

Denote by $G_s(H^2(\g,\KK))$ the Grassmannian consisting of the subspaces of $H^2(\g,\KK)$ of dimension $s$, and introduce an $\Aut(\g)$ action in this space in the obvious way: given $\phi\in\Aut(\g)$ and
\vspace{-0.5cm}
\[
W=\spann\nb{\sqb{\omega_1},\sqb{\omega_2},\ldots, \sqb{\omega_s}} \in G_s(H^2(\g,\KK)),
\]
we define
\[
W\phi=\spann\nb{\sqb{\omega_1\phi},\sqb{\omega_2\phi},\ldots, \sqb{\omega_s\phi}} \in G_s(H^2(\g,\KK)).
\]

\begin{lemma}
	Let $\g$ be a compatible Lie algebra and let
	\begin{align*}
		W_1&=\spann\nb{\sqb{\omega_1},\sqb{\omega_2},\ldots, \sqb{\omega_s}}\in G_s(H^2(\g,\KK)), \\ W_2&=\spann\nb{\sqb{\theta_1},\sqb{\theta_2},\ldots, \sqb{\theta_s}} \in G_s(H^2(\g,\KK)).
	\end{align*}
	If $W_1=W_2$, then 
	\[
	\bigcap_{i=1}^s \ann(\omega_i) \cap Z(\g) = \bigcap_{i=1}^s \ann(\theta_i) \cap Z(\g).
	\]
	
\end{lemma}

\begin{proof}
	If $W_1=W_2$, there is an invertible matrix $(a_{ij})$ such that
	\[
	\sqb{\theta_i} = \sum_{j=1}^s a_{ij}\sqb{\omega_j} \text{, for } i = 1, \ldots, s.
	\]
	Thus, for some maps $\varphi_i:\g \rightarrow \KK$, 
	\[
	{\theta_i} = \sum_{j=1}^s a_{ij} {\omega_j} + \delta\varphi_i \text{, for } i = 1, \ldots, s.
	\]
	Expanding, this means that
	\begin{align*}
		\underline\theta_i(x,y)=\sum_{i=1}^{s} a_{ij}\underline\omega_j(x,y)+\varphi_j({\sqb{x,y}}),
		\undertilde\theta_i(x,y)=\sum_{i=1}^{s} a_{ij}\undertilde\omega_j(x,y)+\varphi_j({\cb{x,y}}).
	\end{align*}
	From this expression it is now easy to see that $\underline\theta_i(x,y)=\undertilde\theta_i(x,y)=\sqb{x,y}=0$ for all $i = 1, \ldots, s$ if and only if $\underline\omega_i(x,y)=\undertilde\omega_i(x,y)=\sqb{x,y}=0$ for all $i = 1, \ldots, s$, and this now allows us to conclude that  indeed 
	\[
	\bigcap_{i=1}^s \ann(\omega_i) \cap Z(\g) = \bigcap_{i=1}^s \ann(\theta_i) \cap Z(\g).
	\]
\end{proof}

From this, the following set is well defined:

\begin{align*}
	\TTT_s(\g):=\Big\{W=\spann\nb{\sqb{\omega_1},\sqb{\omega_2},\ldots, \sqb{\omega_s}} &\in G_s(H^2(\g,\KK)) \Big\vert \bigcap_{i=1}^s \ann(\omega_i) \cap Z(\g)=\cb{0}\Big\}
\end{align*}

Let $\g$ be a compatible Lie algebra and $V$ a vector space of fixed dimension $s$. Define $\mathcal{E}(\g,V)$ to be the set of all full central extensions of $\g$ by an $s$-dimensional vector space $V$ which have no central components or, in more symbolic terms,
\[
\mathcal{E}(\g,V)=\cb{\g_\omega \big\vert \omega \in Z^2(\g,V) \text{ and } \spann\nb{\sqb{\omega_1}, \ldots, \sqb{\omega_s}} \in \TTT_s(\g)}.
\]

Restating Theorem~\ref{maintheorem2} in these terms, we have that $\g_\omega,\g_\theta \in \mathcal{E}(\g,V)$ are isomorphic, if and only if there exists an automorphism $\phi$ of $\g$ such that \[\spann\cb{{\sqb{\omega_i}}\mid 1 \leq s \leq n}=\spann\cb{\sqb{\theta_i\phi}\mid 1 \leq s \leq n} \in \TTT_s(\g).\]

\begin{theorem}
	There exists a one-to-one correspondence between the set of isomorphism classes of $\mathcal{E}(\g,V)$ and the set of $\Aut(\g)$-orbits of $\TTT_s(\g)$.
\end{theorem}

\subsubsection*{The algorithm}

This finally allows us to present an algorithm to construct all the nilpotent compatible Lie algebras up to a target dimension $n$. It proceeds as follows, inductively on $n\geq 1$. We assume we have a list of the isomorphism classes of all nilpotent compatible Lie algebras up to dimension $n-1$.
\begin{itemize}
	\item For each dimension $s$ between $1$ and $n-1$ and for each $(n-s)$-dimensional nilpotent compatible Lie algebra $\g$ from our list, we do the following:
	\begin{itemize}
		\item compute $Z(\g)$, $H^2(\g,\KK)$, $\Aut(\g)$;
		\item compute $\TTT_s(\g)$ and its $\Aut(\g)$-orbits;
		\item construct one extension for each of the orbits obtained in the step above.
	\end{itemize}
	\item We add the $n$-dimensional nilpotent compatible Lie algebras with central components: these are the ones of the form $\g\oplus \KK x$ where $\g$ is an $(n-1)$-dimensional nilpotent compatible Lie algebra.
\end{itemize}

\section{The classification of nilpotent compatible Lie algebras up to dimension \texorpdfstring{$4$}{4}}\label{Classification}

In this section, we use what we have developed in the previous ones to give a complete classification of the nilpotent compatible Lie algebras up to dimension $4$. 

It is useful to introduce the following notation. If $\cb{e_1,\ldots,e_n}$ is a fixed basis for a compatible Lie algebra of dimension $n$, we let $\Delta_{ij}$ be the bilinear map $\g \times \g \rightarrow \KK$ defined on the basis elements by
\[
\Delta_{ij}(e_k,e_\ell)=\begin{cases}
	\:\:\;1, &\text{ if } k=i, \ell=j, \\
	-1, &\text{ if } k=j, \ell=i, \\
	\:\:\; 0, &\text{ otherwise.}
\end{cases}
\]
These maps form a basis of the space $\Hom_\KK\big(\bigwedge^2\g,\KK\big)$.  

\subsection{Some results on automorphisms and cohomology}

We present a few results which will allow us to simplify the practical process of implementing our algorithm. We omit the proofs. 

\begin{proposition}\label{switchaut}
	The following conditions hold.
	\begin{enumerate}[(a)]
		\item If $\g$ and $\hh$ are isomorphic compatible Lie algebras via $\varphi:\g \stackrel{\cong}{\rightarrow} \hh$, then $\Aut(\hh) \cong \Aut(\g)$ via the map $\varphi^{-1}\circ - \circ \varphi$.
		\item If $\g$ and $\hh$ are skew-isomorphic compatible Lie algebras via $\varphi:\g \stackrel{\skiso}{\rightarrow} \hh$, then $\Aut(\hh) \cong \Aut(\g)$ via the map $\varphi^{-1}\circ - \circ \varphi$. In particular, $\Aut(\g) = \Aut(\g^s)$, where $\g^s$ is a switched copy of $\g$.
	\end{enumerate}
\end{proposition}

\begin{proposition}\label{switchcohom}
	The following conditions hold.
	\begin{enumerate}[(a)]
		\item If $\g$ and $\hh$ are isomorphic compatible Lie algebras via $\varphi:\g \stackrel{\cong}{\rightarrow} \hh$, then $Z^2(\hh,V) \cong Z^2(\g,V)$ via the map $\omega \mapsto \omega\varphi$, where $(\omega\varphi)(x,y)=\omega(\varphi(x),\varphi(y))$.
		\item If $\g$ and $\hh$ are skew-isomorphic compatible Lie algebras via the map $\varphi:\g \stackrel{\skiso}{\rightarrow} \hh$, then $Z^2(\hh,V) \cong\nolinebreak Z^2(\g,V)$ via the map $(\underline\omega,\undertilde\omega) \mapsto (\undertilde\omega,\underline\omega)\varphi$. In particular, $(\underline\omega,\undertilde\omega)\in Z^2(\g)$ if and only if $(\undertilde\omega,\underline\omega)\in Z^2(\g^s,V)$, where $\g^s$ is a switched copy of $\g$.
	\end{enumerate}
	Moreover, in each of these settings, $B^2(\hh,V)$ is sent to $B^2(\g,V)$ and thus the above is also valid for $H^2(\hh,V)$ and $H^2(\g,V)$.
\end{proposition}

\begin{proposition}\label{switchext}
	Denoting $\omega^s=(\undertilde\omega,\underline\omega)$ where $\omega = (\underline\omega,\undertilde\omega)$, we have that $(\g^s)_{\omega^s} = (\g_\omega)^s$.
\end{proposition}

\begin{proof}
	By Proposition~\ref{switchcohom}, $\omega^s \in Z^2(\g^s,V)$, and thus $(\g^s)_{\omega^s}$ is well defined. Looking at the definition of the products in an extension via a cocycle, it is immediate to conclude the result.
\end{proof}

\subsection{Dimensions \texorpdfstring{$1$}{1} and \texorpdfstring{$2$}{2}}

In dimension one there is only one compatible Lie algebra up to isomorphism, namely, the abelian one. We denote this isomorphism class by $\NCL_{1}$.

In dimension two, there exists also only one nilpotent compatible Lie algebra up to isomorphism, namely the abelian one, which we denote by $\NCL_{2}$. If a two-dimensional compatible Lie algebra with basis $\cb{x,y}$ has \emph{any} non-zero product, then forcibly one of $\sqb{x,y}$ or $\cb{x,y}$ is non-zero and the algebra is centreless, thus not nilpotent.

\subsection{Dimension \texorpdfstring{$3$}{3}}

To obtain all the nilpotent compatible Lie algebras of dimension $3$, we only need to consider the one-dimensional extensions of $\NCL_2$. This is because the cohomology of $\NCL_1$ is trivial, and thus, there are no two-dimensional extensions of this algebra which suit our algorithm.

Let us now compute the relevant spaces. Let $\cb{e_1,e_2}$ be a basis of $\NCL_2$.
Since $\NCL_2$ is abelian, $Z(\NCL_2)=\NCL_2$ and $\Aut(\NCL_2)=\GL(\NCL_2)$.

Again, because $\NCL_2$ is abelian, the cocycle and cohomology structure is very simple. Since all products are zero, all pairs of alternating bilinear maps $\NCL_2\times\NCL_2 \rightarrow \KK$ are cocycles and $B^2(\NCL_2,\KK)=0$. Thus, we have
\[
H^2(\NCL_2,\KK)=\spann_\KK\cb{\sqb{(\Delta_{12},0)},\sqb{(0,\Delta_{12})}}.
\]
For simplicity, denote $\gamma_1= (\Delta_{12},0)$ and $\gamma_2= (0,\Delta_{12})$.

We note that since we are working in the case $s=1$ and $Z(\NCL_2)=\NCL_2$, $\TTT_1(\NCL_2)$ is simply
\begin{align*}
	\TTT_1(\NCL_2)&=\cb{\spann_\KK(\sqb{\omega})\mid \omega \in H^2(\NCL_2,\KK), \ann(\omega)=0} \\
	&=\cb{\spann_\KK(\sqb{\omega})\mid \omega \in H^2(\NCL_2,\KK),\omega\neq0},
\end{align*}
and any non-zero cohomology class will yield a central extension which has centre of dimension one. (This could also be deduced from the fact that on any compatible Lie algebra the centre has codimension at least two whenever it is not the whole algebra.)

We now describe the action of $\Aut(\NCL_2)$ on $H^2(\NCL_2,\KK)$. Let $\phi \in \Aut(\NCL_2)$. We write
\[
\phi = \nb{\begin{matrix}
		x_{11} & x_{12} \\
		x_{21} & x_{22}
\end{matrix}},
\]
where $\phi(e_j)=\sum_{i}x_{ij}e_i$ and D:= $\det{\phi} \neq 0$.

Let us compute
\begin{align*}
	(\gamma_1 \phi)(e_1,e_2)=\gamma_1(\phi(e_1),\phi(e_2)) = \gamma_1(x_{11}e_1 + x_{21}e_2,x_{12}e_1 + x_{22}e_2) = (D,0).
\end{align*}
We conclude that $\gamma_1 \phi = D \gamma_1$. Likewise, $\gamma_2 \phi = D \gamma_2$.

We now classify the orbits. For an arbitrary cohomology class $\sqb{\omega}=\sqb{\alpha \gamma_1 + \beta \gamma_2}$, one has $(\sqb{\omega})\phi = \sqb{D\omega}$. Considering that $D$ is non-zero, two cohomology classes are in the same $\Aut(\NCL_2)$-orbit if and only if they are multiples of each other. Therefore, we have the following pairwise distinct orbits:
\[
\mathcal{O}(\gamma_1), \qquad \mathcal{O}(\gamma_2), \qquad \mathcal{O}(\gamma_1+\alpha\gamma_2), \qquad \:\alpha \in \KK^\times.
\]

Recalling Definition~\ref{cextdef}, we finish the algorithm by constructing the appropriate extensions of $\NCL_2$. The underlying vector space in each of these is of dimension $3$, so we extend the previous basis by $e_3$. We perform the following extensions:
\begin{itemize}
	\item extending by $\gamma_1$, one obtains $\NCL_{3,2}$, whose only non-zero product is $\sqb{e_1,e_2}=e_3$;
	\item extending by $\gamma_2$, one obtains $\NCL_{3,3}$, whose only non-zero product is $\cb{e_1,e_2}=e_3$;
	\item extending by $\gamma_1+\alpha\gamma_2$, one obtains the one-parameter family $\NCL_{3,4}^\alpha$ (with $\alpha \neq 0$), whose non-zero products are $\sqb{e_1,e_2}=e_3$ and $\cb{e_1,e_2}=\alpha e_3$.
\end{itemize}
We finalise by considering the (only) three-dimensional compatible Lie algebra with central components, namely $\NCL_{3,1}=\NCL_{2}\oplus\KK$, which is just the abelian compatible Lie algebra of dimension $3$.

\subsection{Dimension \texorpdfstring{$4$}{4}}

We divide this section into parts corresponding to which algebra we extend.

\subsubsection*{The algebras with central components}

As has been discussed before, we need to consider the algebras with central components separately. These are obtained by adding a one-dimensional central component to each of the algebras of dimension three. Thus, we have the following:
\begin{align*}
	\NCL_{4,1}=\NCL_{3,1}\oplus\KK, \quad \NCL_{4,2}=\NCL_{3,2}\oplus\KK,  \\	
	\NCL_{4,3}=\NCL_{3,3}\oplus\KK, \quad \NCL_{4,4}^\alpha=\NCL_{3,4}^\alpha\oplus\KK.
\end{align*}
Naturally, the nontrivial relations are the same as the corresponding ones in dimension $3$.

\subsubsection*{The extensions of $\NCL_{3,1}$}

As has been discussed above, for an abelian compatible Lie algebra $\g$, $Z^2(\g,\KK)$ coincides with $\Hom_\KK\big(\bigwedge^2\g,\KK\big)\times\Hom_\KK\big(\bigwedge^2\g,\KK\big)$ and $B^2(\g,\KK)$ is trivial, so we can consider 
\begin{align}\label{gammas}
	\begin{split}
		\gamma_1 = (\Delta_{23},0), \: \gamma_2 = (\Delta_{13},0), \:\gamma_3 = (\Delta_{12},0), \\
		\gamma_4 = (0,\Delta_{23}), \:\gamma_5 = (0,\Delta_{13}), \:\gamma_6 = (0,\Delta_{12}),
	\end{split}
\end{align}
and we have
\[
H^2(\NCL_{3,1})= \spann_\KK\cb{\sqb{\gamma_1},\sqb{\gamma_2},\sqb{\gamma_3},\sqb{\gamma_4},\sqb{\gamma_5},\sqb{\gamma_6}}.
\]
Again, due to the fact that all products are trivial, $\Aut(\NCL_{3,1})$ coincides with $\GL(\NCL_{3,1})$, and an automorphism of $\NCL_{3,1}$ has the form
\[
\phi = \begin{pmatrix}
	x_{11} & x_{12} &x_{13} \\
	x_{21} &x_{22} &x_{23} \\
	x_{31} &x_{32} &x_{33}
\end{pmatrix}, \quad \det(\phi)\neq 0.
\]

Let us compute the action of $\phi=\Aut(\NCL_{3,1})$ on each cocycle. The easiest way is to just consider the basis of the bilinear alternating maps $\cb{\Delta_{23},\Delta_{13},\Delta_{12}}$ and to express each as a matrix:
\[
M(\Delta_{23})=\begin{pmatrix}
	0 & 0 & 0 \\
	0 & 0 & 1 \\
	0 & -1 & 0  
\end{pmatrix}, \:
M(\Delta_{13})=\begin{pmatrix}
	0 & 0 & 1 \\
	0 & 0 & 0 \\
	1 & 0 & 0  
\end{pmatrix}, \:
M(\Delta_{12})=\begin{pmatrix}
	0 & 1 & 0 \\
	-1 & 0 & 0 \\
	0 & 0 & 0  
\end{pmatrix}.
\]
Now, to compute the action, we can do the matrix computations $\phi^t M(\Delta_{ij}) \phi$ to obtain a closed form for the action. This gives the expression
\[
\Delta_{ij} \phi = D_{k1}\Delta_{23} + D_{k2}\Delta_{13} + D_{k3}\Delta_{12}, 
\]
where $(i,j,k)\in\cb{(1,2,3),(2,3,1),(1,3,2)}$,
and $D_{k\ell}$ is the determinant of the $(k,\ell)$-minor of $\phi$.

In terms of the $\gamma_i$, this gives
\begin{align}\label{autactiongamma}
	\begin{split}
		\gamma_1 \phi = D_{11}\gamma_1 + D_{12}\gamma_2 + D_{13}\gamma_3,& 
		\quad  \gamma_4 \phi = D_{11}\gamma_4 + D_{12}\gamma_5 + D_{13}\gamma_6, \\
		\gamma_2 \phi = D_{21}\gamma_1 + D_{22}\gamma_2 + D_{23}\gamma_3,& 
		\quad  \gamma_5 \phi = D_{21}\gamma_4 + D_{22}\gamma_5 + D_{23}\gamma_6, \\
		\gamma_3 \phi = D_{31}\gamma_1 + D_{32}\gamma_2 + D_{33}\gamma_3,& 
		\quad  \gamma_6 \phi = D_{31}\gamma_4 + D_{32}\gamma_5 + D_{33}\gamma_6.		
	\end{split}
\end{align}
Now, to compute on any linear combination of the $\gamma_i$ we simply use linearity.

Finally, before computing the orbits of the action of $\Aut(\NCL_{3,1})$ on $H^2(\NCL_{3,1})$, we need to check which cocycles $\omega$ satisfy the condition $\ann(\omega) \cap Z(\NCL_{3,1}) = 0$. 

But since $Z(\NCL_{3,1})= \NCL_{3,1}$, these are simply the cocycles with trivial annihilator.
After solving a system of linear equations, one concludes that the admissible cocycles $\omega=(\underline\omega,\undertilde\omega)$ are the ones where both $\underline\omega$ and $\undertilde\omega $ are nonzero and distinct.

\newpage

\textbf{Computing the orbits}

We will not show all the computations, preferring to outline the process and present just one illustrative example. We leave the full computations for the next case ($\NCL_{3,2}$).

With the restrictions on admissible cocycles, one can conclude that we can take $\underline{\omega}$ to be of the form $\underline\omega=\gamma_1$, since we can find an automorphism of $\NCL_{3,1}$ that maps any nonzero linear combination $\alpha\gamma_1 + \beta\gamma_2+\delta\gamma_3$ to $\gamma_1$, thus establishing that those are in the same $\Aut(\NCL_{3,1})$-orbit.

Following that, 
one can further conclude that there is an automorphism of $\NCL_{3,1}$ that fixes $\Delta_{2,3}$ and maps any nonzero alternating bilinear map different from $\Delta_{2,3}$ to any other in the same conditions.
%
Thus, any cocycle of the form $\gamma_1+(0,\undertilde\omega)$, with $\undertilde\omega \neq 0$ and $ \undertilde\omega \neq\Delta_{23}$ can be mapped to any other of the same form, whence only one orbit appears for the action, and thus there is only one central extension of $\NCL_{3,1}$ whose centre has dimension exactly one.
We will choose the representative $\omega = \gamma_1 + \gamma_5$, and so we obtain a $4$-dimensional compatible Lie algebra with relations
\[
\sqb{e_2,e_3}=e_4 \quad \cb{e_1,e_3}=e_4,
\]
which we denote by $\NCL_{4,5}$.

\begin{example}	
	As an illustrative example (checking with the expressions in~\eqref{autactiongamma}), we can see that the automorphism
	\[
	\phi = \begin{pmatrix}
		1/\beta' & 0&0 \\
		0 &1 & 0 \\
		0 & 0 & 1
	\end{pmatrix},
	\]
	sends $\gamma_1+\beta'\gamma_5$ to $\gamma_1+\gamma_5$, and that the automorphism
	\[
	\phi = \begin{pmatrix}
		\delta & 0 & 0 \\
		0 & 0 & -1 \\
		0 & 1 & 0
	\end{pmatrix},
	\]
	sends $\gamma_1+\gamma_5$ to $\gamma_1+\delta'\gamma_6$. In the same way, all other admissible cocycles can be mapped to $\omega = \gamma_1 + \gamma_5$.
\end{example}

\subsubsection*{The extensions of $\NCL_{3,2}$}

As with the previous case, we compute $Z^2(\NCL_{3,2},\KK)$, $H^2(\NCL_{3,2},\KK)$ and $\Aut(\NCL_{3,2})$. We recall that the only non-zero product is $\sqb{e_1,e_2}=e_3$.

It is not too hard to verify that $Z^2(\NCL_{3,2},\KK)$ is still just the space of pairs of alternating bilinear forms. 

%

Let $\varphi\in {\NCL_{3,2}}^*$ and set $\lambda=\varphi(e_3)$. Then 
\begin{align*}
	&\delta\underline\varphi(e_1,e_2)= \varphi(\sqb{e_1,e_2})=\varphi(e_3)=\lambda, 
\end{align*}
and $\delta\underline\varphi(e_i,e_j)=0$ for all other pairs $i<j$. Thus $B^2(\NCL_{3,2},\KK)=\spann_\KK\cb{\gamma_3}$, and \[H^2(\NCL_{3,2},\KK)=\spann_\KK\cb{\sqb{\gamma_1},\sqb{\gamma_2},\sqb{\gamma_4},\sqb{\gamma_5},\sqb{\gamma_6}},\]
using the same notation as in~\eqref{gammas}.

We now focus our attention on $\Aut(\NCL_{3,2})$. Let $\phi\in\Aut(\NCL_{3,2})$, with $\phi(e_j)=\sum_i x_{ij}e_i$. We have that
\begin{align*}
	\phi(e_3)&= \phi(\sqb{e_1,e_2}) = \sqb{\phi(e_1),\phi(e_2)} \\
	&= \sqb{x_{11}e_1+x_{21}e_2+x_{31}e_3,\:x_{12}e_1+x_{22}e_2+x_{32}e_3} \\
	&= (x_{11}x_{22}-x_{12}x_{21})e_3,
\end{align*}
with no further restrictions, since there are no other nonzero products to consider. Again, denoting by $D_{ij}$ the determinant of the $(i,j)$-minor of the matrix representing $\phi$, we have that
\[\phi = \begin{pmatrix}
	x_{11} & x_{12} & 0 \\
	x_{21} &x_{22} & 0 \\
	x_{31} &x_{32} &D_{33}
\end{pmatrix}, \quad D_{33}\neq 0.
\]
The action by automorphisms still follows the expressions in~\eqref{autactiongamma}, whence, for an automorphism as above, by simplifying the expression taking into account the entries which are equal to zero, we obtain
\begin{align}\label{orbitscocycle32}
	\begin{split}
		\gamma_1 \phi &= x_{22}D_{33}\gamma_1 + x_{21}D_{33}\gamma_2 + D_{13}\gamma_3,  \\
		\gamma_2 \phi &= x_{12}D_{33}\gamma_1 + x_{11}D_{33}\gamma_2 + D_{23}\gamma_3, \\
		\gamma_3 \phi &= D_{33}\gamma_3, 		
	\end{split}
	\begin{split}
		\gamma_4 \phi &= x_{22}D_{33}\gamma_4 + x_{21}D_{33}\gamma_5 + D_{13}\gamma_6,  \\
		\gamma_5 \phi &= x_{12}D_{33}\gamma_4 + x_{11}D_{33}\gamma_5 + D_{23}\gamma_6, \\
		\gamma_6 \phi &= D_{33}\gamma_6. 		
	\end{split}
\end{align}
At the cohomology level, we have the following
{\small
	\begin{align}\label{orbitscohom32}
		\begin{split}
			\sqb{\gamma_1} \phi &= x_{22}D_{33}\sqb{\gamma_1} + x_{21}D_{33}\sqb{\gamma_2},\\
			\sqb{\gamma_2} \phi &= x_{12}D_{33}\sqb{\gamma_1} + x_{11}D_{33}\sqb{\gamma_2}, \\
			\: 		
		\end{split}
		\begin{split}
			\sqb{\gamma_4} \phi &= x_{22}D_{33}\sqb{\gamma_4} + x_{21}D_{33}\sqb{\gamma_5} + D_{13}\sqb{\gamma_6},  \\
			\sqb{\gamma_5} \phi &= x_{12}D_{33}\sqb{\gamma_4} + x_{11}D_{33}\sqb{\gamma_5} + D_{23}\sqb{\gamma_6}, \\
			\sqb{\gamma_6} \phi &= D_{33}\sqb{\gamma_6}. 		
		\end{split}
\end{align}}

Again, the final step before computing the orbits of the $\Aut(\NCL_{3,2})$-action on $H^2(\NCL_{3,2})$ is to check which cocycles $\omega$ satisfy the condition $\ann(\omega) \cap Z(\NCL_{3,2}) = 0$. 

Since $Z(\NCL_{3,2})= \spann_\KK\cb{e_3}$, we are looking for the ones whose annihilator does not contain $e_3$. 
Since $\gamma_1(e_2,e_3)=(1,0)$, $\gamma_2(e_1,e_3)=(1,0)$, $\gamma_4(e_2,e_3)=(0,1)$, and $\gamma_5(e_1,e_3)=(0,1)$, the cocycles satisfying this restriction are precisely those where any one of $\alpha$, $\beta$, $\alpha'$ or $\beta'$ is nonzero, if we write
\[
\omega = \alpha\gamma_1 + \beta\gamma_2 + \delta\gamma_3 + \alpha'\gamma_4 + \beta'\gamma_5 + \delta'\gamma_6.
\]
In other words, the ones that are not admissible are those of the form $\omega=\delta\gamma_3+\delta'\gamma_6$.

\textbf{Computing the orbits}

In what follows, we will compute a complete list of representatives of the orbits of the action on the admissible cocycles. 
To simplify the computations, we will use the expressions in~\eqref{orbitscohom32} and compute the orbits directly in the cohomology space.

Let then
\[
\sqb{\omega} = \sqb{\alpha\gamma_1 + \beta\gamma_2  + \alpha'\gamma_4 + \beta'\gamma_5 + \delta'\gamma_6}.
\]

As before, we will focus on $\underline{\omega}$ first, and we will show that any cohomology class is in the same orbit as one of the form $\sqb{(0,\undertilde\omega)}$ or $\sqb{\gamma_1 +(0, \undertilde\omega)}$ (that is, where $\underline\omega=0$ or $\underline\omega=\gamma_1$).

Writing $(\underline\omega,0) = \alpha\gamma_1+\beta\gamma_2$, we have the following.

If $\beta \neq 0$ and $\alpha = 0$, consider the automorphism
\[
\phi = \begin{pmatrix}
	0 & -\beta &0 \\
	1 &0 & 0 \\
	0 & 0 & \beta
\end{pmatrix},  	\text{ which sends $\gamma_1$ to $\beta\gamma_2$.}
\]

If $\alpha \neq 0$, consider the automorphism
\[
\phi = \begin{pmatrix}
	\alpha & 0& 0 \\
	\beta/\alpha & 1 & 0 \\
	0 & 0 & \alpha
\end{pmatrix}, 	\text{ which sends $\gamma_1$ to $\alpha\gamma_1+\beta\gamma_2$.}
\]

So, any cohomology class of the form $\sqb{\alpha\gamma_1+\beta\gamma_2 + (0,\undertilde\omega)}$ is in the same orbit as one of the form $\sqb{\gamma_1 + (0,\undertilde\omega')}$ or one of the form $\sqb{(0,\undertilde\omega')}$. We now distinguish these two cases separately.

\underline{Case $\underline\omega=0$}:

Write $\omega=\alpha'\gamma_4+\beta'\gamma_5+\delta'\gamma_6$.

%
%
%
%
%
%
%
%
%

If $\alpha' = 0$, then $\beta' \neq 0$. Consider the automorphism
\[
\phi = \begin{pmatrix}
	0 & -\beta'&0 \\
	1 & 0 & 0 \\
	1 & \delta' & \beta'
\end{pmatrix}, 	\text{ which sends $\gamma_4$ to $\beta'\gamma_5+\delta'\gamma_6$.}
\]

If $\alpha'\neq 0$, consider the automorphism
\[
\phi = \begin{pmatrix}
	\alpha' &0 &0 \\
	\beta'/\alpha' & 1 & 0 \\
	-\delta' & 0 & \alpha'
\end{pmatrix}, 	\text{ which sends $\gamma_4$ to $\alpha'\gamma_4+\beta'\gamma_5+\delta'\gamma_6$.}
\]

Thus, all cocycles with $\underline{\omega}=0$ are in the same orbit, and we may choose the representative $\omega=\gamma_4$.

\underline{Case $\underline\omega\neq0$}: As seen before, all cocycles $\omega$ where $\underline\omega\neq0$ are in the same orbit as one with $\underline\omega=\gamma_1$, so we can assume $\underline\omega=\gamma_1$. 
So let $\omega = \gamma_1+\alpha'\gamma_4+\beta'\gamma_5+\delta'\gamma_6$.

If $\alpha',\beta',\delta' = 0$, we have simply $\omega=\gamma_1$.

If $\alpha' \neq 0$ and $\beta',\delta' = 0$, we cannot reduce this any further (this fact will be proven later, when we prove that the orbits we obtain here are distinct) and thus we have the representative $\gamma_1+\alpha'\gamma_4$.


If $\delta' \neq 0$ and $\alpha',\beta' = 0$, consider the automorphism
\[
\phi = \begin{pmatrix}
	\delta'^2 &  0 &0 \\
	0 &1/\delta' & 0 \\
	0 & 0 & \delta'
\end{pmatrix}, \text{ which sends $\gamma_1+\gamma_6$ to $\gamma_1+\delta'\gamma_6$.}
\]

%

If $\alpha',\delta' \neq 0$ and $\beta' = 0$, consider the automorphism
\[
\phi = \begin{pmatrix}
	1 & 0&0 \\
	0 & 1 & 0 \\
	\delta'/\alpha' & 0 & 1
\end{pmatrix}, \text{ which sends $\gamma_1+\alpha'\gamma_4+\delta' \gamma_6$ to $\gamma_1+\gamma_3+\alpha'\gamma_4$.}
\]
In terms of cohomology this yields $\sqb{\gamma_1+\alpha'\gamma_4+\delta' \gamma_6} \in \mathcal{O}(\sqb{\gamma_1+\alpha'\gamma_4})$.


If $\beta' \neq 0$, consider the automorphism
\[
\phi = \begin{pmatrix}
	1 & -\alpha'/\beta' &0 \\
	0 &1 & 0 \\
	0 & -\delta'/\beta' & 1\
\end{pmatrix}, 	\text{ which sends $\gamma_1+\alpha'\gamma_4+\beta'\gamma_5+\delta' \gamma_6$ to $\gamma_1+\beta'\gamma_5$.}
\]

We thus have all cases covered and we have the following orbits (already taking cohomology into account): 
\begin{equation}\label{orbits32}
	\mathcal{O}(\sqb{\gamma_1}), \quad\! \mathcal{O}(\sqb{\gamma_4}), \quad\! \mathcal{O}(\sqb{\gamma_1+\beta\gamma_4}), \quad\! \mathcal{O}(\sqb{\gamma_1+\beta\gamma_5}), \quad\! \mathcal{O}(\sqb{\gamma_1+\gamma_6}), \quad\! \beta \in \KK^\times
\end{equation}
where the third and fourth orbits are one-parameter families. We will now show that all of these orbits are pairwise distinct apart from different orbits in the family $\mathcal{O}(\sqb{\gamma_1+\beta\gamma_5})$. This case will be examined separately and we will be able to show that some will in fact coincide depending on the parameter $\beta$.

Firstly, it is clear that $\mathcal{O}(\sqb{\gamma_1})$ and $ \mathcal{O}(\sqb{\gamma_4})$ are distinct from one another and from the remaining orbits, as the action cannot relate two cohomology classes $\sqb{\omega}$ and $\sqb{\theta}$ if $\underline\omega=0$ and $\underline\theta\neq0$ or if $\undertilde\omega=0$ and $\undertilde\theta\neq0$.

It is also clear that $\mathcal{O}(\sqb{\gamma_1+\gamma_6})$ is distinct from the others, as the action on $\gamma_6$ cannot relate it to any element where the coefficients of $\gamma_4$ or $\gamma_5$ are nonzero.

To show that $\mathcal{O}(\sqb{\gamma_1+\beta\gamma_4})\neq \mathcal{O}(\sqb{\gamma_1+\beta'\gamma_4})$ for $\beta'\neq\beta$ we just note that if
$\sqb{\gamma_1+\beta\gamma_4}\phi = \sqb{\gamma_1+\beta'\gamma_4}$ then $\sqb{\gamma_1}\phi=\sqb{\gamma_1}$ and $\sqb{\beta \gamma_4}\phi=\sqb{\beta' \gamma_4}$.
According to~\eqref{orbitscohom32} we have that $x_{22}D_{33}=1$ and $x_{21}=0$. Thus 
\[
\sqb{\gamma_1+\beta'\gamma_4}=\sqb{\gamma_1+\beta\gamma_4}\phi =\sqb{\gamma_1+\beta\gamma_4+ \beta D_{13}\gamma_6},
\]
and it follows that $\beta'=\beta$. 

Lastly, to show that the orbits $\mathcal{O}(\sqb{\gamma_1+\beta\gamma_4})$ and $ \mathcal{O}(\sqb{\gamma_1+\beta'\gamma_5})$ are distinct for $\beta\beta'\neq 0$, suppose that there exists an automorphism $\phi$ such that $\sqb{\gamma_1+\beta\gamma_4}\phi=\sqb{\gamma_1+\beta'\gamma_5}$. 
Then, the coefficient of $\sqb{\gamma_2}$ in $\sqb{\gamma_1}\phi$ is equal to zero, while the coefficient of $\sqb{\gamma_5}$ in $\sqb{\gamma_4}\phi$ is equal to $\beta'/\beta$. But both coefficients are equal to $x_{21}D_{33}$, which is a contradiction.


We now focus on determining when $\mathcal{O}(\sqb{\gamma_1+\beta'\gamma_5})=\mathcal{O}(\sqb{\gamma_1+\beta\gamma_5})$ for $\beta\beta'\neq 0$.

Let $\phi$ be an arbitrary automorphism of $\NCL_{3,2}$, and write it as	
\[\phi = \begin{pmatrix}
	x_{11} & x_{12} & 0 \\
	x_{21} &x_{22} & 0 \\
	x_{31} &x_{32} &D_{33}
\end{pmatrix}, \quad D_{33}\neq 0.
\]
Suppose that $\sqb{\gamma_1+\beta\gamma_5}\phi=\sqb{\gamma_1+\beta'\gamma_5}$. Using the expressions in~\eqref{orbitscohom32} we obtain the following system of equations
\begin{equation*}
	\begin{cases}
		x_{22}D_{33}=1 \\
		x_{21}D_{33}=0 \\
		x_{12}D_{33}=0 \\
		\beta x_{11}D_{33}=\beta' \\
		\beta(x_{11}x_{32}-x_{31}x_{12})=0.
	\end{cases}
\end{equation*}
It can be easily verified that this system of equations is possible if and only if $\beta'/\beta$ has a cube root in $\KK$. In particular, we only have one orbit if and only if the field $\KK$ is cubically closed.

Thus, we have obtained all the distinct $\Aut(\NCL_{3,2})$-orbits. Using the same representatives as in~\eqref{orbits32} we obtain the following extensions.

\begin{itemize}
	\item $\NCL_{4,6}$, with relations $\sqb{e_1,e_2}=e_3$, $\sqb{e_2,e_3}=e_4$; 
	\item $\NCL_{4,7}$, with relations $\sqb{e_1,e_2}=e_3$, $\cb{e_2,e_3}=e_4$;
	\item $\NCL_{4,8}^{\beta}$, with relations $\sqb{e_1,e_2}=e_3$, $\sqb{e_2,e_3}=e_4$, $\cb{e_2,e_3}=\beta e_4$ ($\beta'\neq 0$); 
	\item $\NCL_{4,9}^{\beta}$, with relations $\sqb{e_1,e_2}=e_3$, $\sqb{e_2,e_3}=e_4$, $\cb{e_1,e_3}= \beta e_4$ ($\beta'\neq 0$), where $\NCL_{4,9}^{\beta}\cong\nolinebreak\NCL_{4,9}^{\beta'}$ if and only if $\beta'/\beta$ has a cube root in $\KK$;
	\item $\NCL_{4,10}$, with relations $\sqb{e_1,e_2}=e_3$, $\sqb{e_2,e_3}=e_4$,  $\cb{e_1,e_2}= e_4$.
\end{itemize}

\subsubsection*{The extensions of $\NCL_{3,3}$}

Since $\NCL_{3,3}=(\NCL_{3,2})^s$, we can use Propositions~\ref{switchaut} -- \ref{switchext} 
to obtain the list of one-dimensional extensions of $\NCL_{3,3}$ by just considering the switched copies of the algebras in the corresponding list for $\NCL_{3,2}$. 

They are the following:

\begin{itemize}
	\item $\NCL_{4,11}$, with relations $\cb{e_1,e_2}=e_3$, $\cb{e_2,e_3}=e_4$; 
	\item $\NCL_{4,12}$, with relations $\sqb{e_2,e_3}=e_4$, $\cb{e_1,e_2}=e_3$;
	\item $\NCL_{4,13}^{\beta}$, with relations $\sqb{e_2,e_3}=\beta e_4$, $\cb{e_1,e_2}=e_3$, $\cb{e_2,e_3}=e_4 \quad$ ($\beta'\neq 0$); 
	\item $\NCL_{4,14}^{\beta}$, with relations $\sqb{e_1,e_3}= \beta e_4$, $\cb{e_1,e_2}=e_3$, $\cb{e_2,e_3}=e_4 \quad$ ($\beta'\neq 0$), where $\NCL_{4,14}^{\beta}\cong \NCL_{4,14}^{\beta'}$ if and only if $\beta'/\beta$ has a cube root in $\KK$;
	\item $\NCL_{4,15}$, with relations $\sqb{e_1,e_2}= e_4$, $\cb{e_1,e_2}=e_3$, $\cb{e_2,e_3}=e_4$.
\end{itemize}

\subsubsection*{The extensions of $\NCL_{3,4}^\alpha$}

Again, we start by computing the spaces $Z^2(\NCL_{3,4}^\alpha,\KK)$, $H^2(\NCL_{3,4}^\alpha,\KK)$ and $\Aut(\NCL_{3,4}^\alpha)$, where $\alpha \in \KK^\times$. We recall that the non-zero products are $\sqb{e_1,e_2}=e_3$ and $\cb{e_1,e_2}=\alpha e_3$.

Once more, we have that $Z^2(\NCL_{3,4}\alpha,\KK)$ coincides with the space of pairs alternating bilinear functions. In a similar way as before, we can also conclude that $B^2(\NCL_{3,2},\KK)=\spann_\KK\cb{\gamma_3+\alpha\gamma_6}$. Thus, 
\[H^2(\NCL_{3,4}\alpha,\KK)=\spann_\KK\cb{\sqb{\gamma_1},\sqb{\gamma_2},\sqb{\gamma_4},\sqb{\gamma_5},\sqb{\gamma_6}},\]
using the same notation as in~\eqref{gammas}.

We now focus our attention on $\Aut(\NCL_{3,4}^\alpha)$. Let $\phi\in\Aut(\NCL_{3,4}^\alpha)$, with $\phi(e_j)=\sum_i x_{ij}e_i$.
We have that
\begin{align*}
	\phi(e_3)&= \phi(\sqb{e_1,e_2}) = \sqb{\phi(e_1),\phi(e_2)} \\
	&= \sqb{x_{11}e_1+x_{21}e_2+x_{31}e_3,\:x_{12}e_1+x_{22}e_2+x_{32}e_3} \\
	&= (x_{11}x_{22}-x_{12}x_{21})e_3,
\end{align*}
and
\begin{align*}
	\alpha\phi(e_3)&= \phi(\cb{e_1,e_2}) = \cb{\phi(e_1),\phi(e_2)} \\
	&= \cb{x_{11}e_1+x_{21}e_2+x_{31}e_3,\:x_{12}e_1+x_{22}e_2+x_{32}e_3} \\
	&= \alpha(x_{11}x_{22}-x_{12}x_{21})e_3,
\end{align*}
with no further restrictions, since there are no other nonzero products to consider. Again, denoting by $D_{ij}$ the determinant of the $(i,j)$-minor of the matrix representing $\phi$, we have that
\[\phi = \begin{pmatrix}
	x_{11} & x_{12} & 0 \\
	x_{21} &x_{22} & 0 \\
	x_{31} &x_{32} &D_{33}
\end{pmatrix}, \quad D_{33}\neq 0,
\]
as in the case for $\Aut(\NCL_{3,2})$. Therefore, the action on the cocycles still has the expression~\eqref{orbitscocycle32}, 
but now $B^2(\NCL_{3,4}\alpha,\KK)$ is generated by $\gamma_3+\alpha\gamma_6$. Thus, $\sqb{\gamma_3}=\sqb{-\alpha\gamma_6}=-\alpha\sqb{\gamma_6}$.

\textbf{Computing the orbits}

Using the same arguments as in the case of $\NCL_{3,2}$, we can obtain the same orbits in cohomology but for one pair. 

The remaining case to consider is when relating $\sqb{\gamma_1+\alpha'\gamma_4+\delta' \gamma_6}$ and $\sqb{\gamma_1+\alpha'\gamma_4}$, as that depended on passing to the cohomology. Despite that, these elements are still on the same orbit, where this time we consider the automorphism
\[
\phi = \begin{pmatrix}
	(\alpha-\alpha')^2/\delta'^2 & 0&0 \\
	0 & \delta'/(\alpha-\alpha') & 0 \\
	-	(\alpha-\alpha')/\delta' & 0 & 	(\alpha-\alpha')/\delta'
\end{pmatrix},
\]
which sends $\gamma_1+\alpha'\gamma_4+\delta' \gamma_6$ to $\gamma_1+\gamma_3+\alpha'\gamma_4+\alpha\gamma_6$. 

\noindent Since $\sqb{\gamma_3+\alpha\gamma_6}=0$, this yields $\sqb{\gamma_1+\alpha'\gamma_4+\delta' \gamma_6} \in \mathcal{O}(\sqb{\gamma_1+\alpha'\gamma_4})$.

We consider the same orbits as before, namely
\begin{equation}
	\mathcal{O}(\sqb{\gamma_1}), \quad \mathcal{O}(\sqb{\gamma_4}), \quad \mathcal{O}(\sqb{\gamma_1+\beta\gamma_4}), \quad \mathcal{O}(\sqb{\gamma_1+\beta\gamma_5}), \quad \mathcal{O}(\sqb{\gamma_1+\gamma_6}), \quad \beta \neq 0.
\end{equation}

The arguments for the first four orbits being pairwise distinct remain correct, as we can argue about the coefficients in the action on the cocycles. But now, due to the different cohomology structure, there is a collapse and actually, the last orbit coincides with the first. To see that, we consider the automorphism
\[
\phi = \begin{pmatrix}
	1 & 0&0 \\
	0 & 1 & 0 \\
	-1/\alpha & 0 & 1
\end{pmatrix}.
\]
It sends $\gamma_1+ \gamma_6$ to $\gamma_1+(1/\alpha)\gamma_3+\gamma_6$. Since $\sqb{\gamma_3+\alpha\gamma_6}=0$, on the cohomology level this yields $\sqb{\gamma_1+\gamma_6} \in \mathcal{O}(\sqb{\gamma_1})$.

Thus, we have the following \emph{distinct} orbits
\begin{equation}
	\mathcal{O}(\sqb{\gamma_1}), \quad \mathcal{O}(\sqb{\gamma_4}), \quad \mathcal{O}(\sqb{\gamma_1+\beta\gamma_4}), \quad \mathcal{O}(\sqb{\gamma_1+\beta\gamma_5}), \quad \beta \neq 0,
\end{equation}
and using these representatives, we obtain the following extensions
\begin{itemize}
	\item $\NCL_{4,16}^\alpha$, with relations $\sqb{e_1,e_2}=e_3$, $\sqb{e_2,e_3}=e_4$, $\cb{e_1,e_2}=\alpha e_3$; 
	\item $\NCL_{4,17}^\alpha$, with relations $\sqb{e_1,e_2}=e_3$, $\cb{e_1,e_2}=\alpha e_3$, $\cb{e_2,e_3}=e_4$;
	\item $\NCL_{4,18}^{\alpha,\beta}$, with relations $\sqb{e_1,e_2}=e_3$, $\sqb{e_2,e_3}=e_4$, $\cb{e_1,e_2}=\alpha e_3$, $\cb{e_2,e_3}=\beta e_4 \quad $ ($\beta \neq 0$); 
	\item $\NCL_{4,19}^{\alpha,\beta}$, with relations $\sqb{e_1,e_2}=e_3$, $\sqb{e_2,e_3}=e_4$, $\cb{e_1,e_2}=\alpha e_3$, $\cb{e_1,e_3}= \beta e_4$, where $\NCL_{4,19}^{\alpha,\beta}\cong \NCL_{4,19}^{\alpha',\beta'}$ if and only if $\alpha=\alpha'$ and $\beta'/\beta$ has a cube root in $\KK$.
\end{itemize}
\subsubsection*{The extensions of $\NCL_{1}$ and $\NCL_{2}$}

We now focus our attention on the $3$-dimensional extensions of $\NCL_{1}$ and the $2$-dimensional extensions of $\NCL_{2}$. 

Considering that $H^2(\NCL_{1},\KK)$ is trivial and that $H^2(\NCL_{2},\KK)$ is one-dimensional, we have that both $G_3(H^2(\NCL_{1},\KK))$ and $G_2(H^2(\NCL_{2},\KK))$ are empty. Thus there are no extensions of these algebras in dimension $4$ to be considered by our algorithm.

\subsection{The complete classification table}

Here, we present a table with all the nilpotent compatible Lie algebras of dimension at most $4$ over a field $\KK$ of characteristic not $2$, up to isomorphism. We denote each isomorphism class by $\NCL_{n,q}^{par}$, where $n$ is the dimension, and $par$ is a (possibly empty) set of non-zero\footnote{Every time we make a parameter equal to zero we recover an algebra with one less parameter already on the list.} parameters.

The table presents the non-zero products between the basis elements $e_1, \ldots, e_n$.

{\arraycolsep=1pt\def\arraystretch{2.2}
	\begin{longtable}{ | c | c| c | } 
		\hline
		$\text{Algebra}$ & $\text{Relations}$ & $\parbox{4.6em}{Dimension of centre}$\\
		\hhline{|=|=|=|}
		$\NCL_{1} $& $\text{Abelian of dimension } 1$ & $1$ \\ 
		\hline
		$\NCL_{2}$ & $\text{Abelian of dimension } 2$ & $2$ \\ 
		\hline
		$\NCL_{3,1}$ & $\text{Abelian of dimension } 3$ & $3$\\ 
		\hline
		$\NCL_{3,2}$ & $\sqb{e_1,e_2}=e_3$& $1$ \\ 
		\hline
		$\NCL_{3,3}$ & $\cb{e_1,e_2}=e_3$ & $1$\\ 
		\hline
		$\NCL_{3,4}^\alpha$ & $\sqb{e_1,e_2}=e_3$\quad $\cb{e_1,e_2}=\alpha e_3 $& $1$ \\ 
		\hline 
		$\NCL_{4,1}$ & $\text{Abelian of dimension } 4$ & $4$\\ 
		\hline
		$\NCL_{4,2}$ & $\sqb{e_1,e_2}=e_3$ & $2$\\ 
		\hline
		$\NCL_{4,3}$ & $\cb{e_1,e_2}=e_3$ & $2$\\ 
		\hline
		$\NCL_{4,4}^\alpha$ & $\sqb{e_1,e_2}=e_3$\quad $\cb{e_1,e_2}=\alpha e_3$ & $2$ \\ 
		\hline
		$\NCL_{4,5}$ & $\sqb{e_2,e_3}=e_4$ \quad $\cb{e_1,e_3}=e_4$ & $1$\\ 
		\hline 
		$\NCL_{4,6}$ & $\sqb{e_1,e_2}=e_3$ \quad $\sqb{e_2,e_3}=e_4$& $1$ \\ 
		\hline
		$\NCL_{4,7}$ & $\sqb{e_1,e_2}=e_3$  \quad$\cb{e_2,e_3}=e_4$ & $1$\\
		\hline
		$\NCL_{4,8}^{\beta}$ & $\sqb{e_1,e_2}=e_3$  \quad $\sqb{e_2,e_3}=e_4$ \quad   $\cb{e_2,e_3}=\beta e_4$ & $1$\\ 
		\hline
		$\NCL_{4,9}^{\beta}$ & $\sqb{e_1,e_2}=e_3$  \quad $\sqb{e_2,e_3}=e_4$ \quad  $\cb{e_1,e_3}= \beta e_4$ & $1$\\ 
		\hline	
		$\NCL_{4,10}$ & $\sqb{e_1,e_2}=e_3$  \quad $\sqb{e_2,e_3}=e_4$ \quad  $\cb{e_1,e_2}= e_4$ & $1$\\ 
		\hline 
		$\NCL_{4,11}$ & $\cb{e_1,e_2}=e_3$ \quad $\cb{e_2,e_3}=e_4$ & $1$\\ 
		\hline
		$\NCL_{4,12}$ & $\cb{e_1,e_2}=e_3$  \quad$\sqb{e_2,e_3}=e_4$ & $1$\\
		\hline
		$\NCL_{4,13}^{\beta}$ & $\cb{e_1,e_2}=e_3$  \quad $\cb{e_2,e_3}=e_4$ \quad   $\sqb{e_2,e_3}=\beta e_4$ & $1$\\ 
		\hline
		$\NCL_{4,14}^{\beta}$ & $\cb{e_1,e_2}=e_3$  \quad $\cb{e_2,e_3}=e_4$ \quad  $\sqb{e_1,e_3}= \beta e_4$& $1$\\ 
		\hline	
		$\NCL_{4,15}$ & $\cb{e_1,e_2}=e_3$  \quad $\cb{e_2,e_3}=e_4$ \quad  $\sqb{e_1,e_2}= e_4$& $1$\\ 
		\hline 
		$\NCL_{4,16}^\alpha$ & $\sqb{e_1,e_2}=e_3$ \quad $\sqb{e_2,e_3}=e_4$ \quad $\cb{e_1,e_2}=\alpha e_3$ & $1$\\ 
		\hline
		$\NCL_{4,17}^\alpha$ & $\sqb{e_1,e_2}=e_3$  \quad $\cb{e_1,e_2}=\alpha e_3$ \quad $\cb{e_2,e_3}=e_4$ & $1$\\
		\hline
		$\NCL_{4,18}^{\alpha,\beta}$ & $\begin{matrix}
			\sqb{e_1,e_2}=e_3  & \sqb{e_2,e_3}=e_4 \\ \cb{e_1,e_2}=\alpha e_3 & \quad \cb{e_2,e_3}=\beta e_4
		\end{matrix}$ & $1$\\ 
		\hline
		$\NCL_{4,19}^{\alpha,\beta}$ & $\begin{matrix}
			\sqb{e_1,e_2}=e_3  & \sqb{e_2,e_3}=e_4 \\ \cb{e_1,e_2}=\alpha e_3 & \quad \cb{e_1,e_3}= \beta e_4
		\end{matrix}$ & $1$\\  
		\hline	
	\end{longtable}
	
}
The only possible isomorphisms from this table are the following:
\begin{itemize}
	\item $\NCL_{4,9}^{\beta}\cong \NCL_{4,9}^{\beta'}$ if and only if $\beta'/\beta$ has a cube root in $\KK$;
	\item $\NCL_{4,14}^{\beta}\cong \NCL_{4,14}^{\beta'}$ if and only if $\beta'/\beta$ has a cube root in $\KK$;
	\item $\NCL_{4,19}^{\alpha,\beta}\cong \NCL_{4,19}^{\alpha',\beta'}$ if and only if $\alpha'=\alpha$ and $\beta'/\beta$ has a cube root in $\KK$.
\end{itemize}
In particular, if the base field $\KK$ is cubically closed, each of the first two of these three families reduces to one isomorphism class, ant the third to a one-parameter family.

\subsection{Skew-isomorphisms within the classification}

The above table does not have isomorphic pairs of algebras, by construction. Nonetheless, two compatible Lie algebras can be skew-isomorphic without being isomorphic, and this section is dedicated to presenting the skew-isomorphisms that appear.

By the symmetry of the definition of nilpotency, we can immediately see that a compatible Lie algebra $\g$ is nilpotent if and only if $\g^s$ is. So, all algebras in the above classification come in pairs of skew-isomorphic copies (where it may be possible  for an algebra to be paired up with itself).

Up to dimension $3$, we have the following pairings.

\begin{itemize}
	\item $\NCL_1$, $\NCL_2$ and $\NCL_{3,1}$ are self skew-isomorphic,
	\item $\NCL_{3,2}\skiso\NCL_{3,3}$ (with the above relations, $(\NCL_{3,2})^s=\NCL_{3,3}$),
	\item $\NCL_{3,4}^\alpha\skiso \NCL_{3,4}^{1/\alpha}$.
\end{itemize}

In dimension $4$ we have the following skew-isomorphisms:
\begin{itemize}
	\item $\NCL_{4,1}$ and $\NCL_{4,5}$ are self skew-isomorphic,
	\item $\NCL_{4,2}\skiso\NCL_{4,3}$ (with the above relations, $(\NCL_{4,2})^s=\NCL_{4,3}$),
	\item $\NCL_{4,4}^\alpha\skiso \NCL_{4,4}^{1/\alpha}$,
	\item $\NCL_{4,6}\skiso \NCL_{4,11}$ (with the above relations, $(\NCL_{4,6})^s=\NCL_{4,11}$),
	\item $\NCL_{4,7}\skiso \NCL_{4,12}$ (with the above relations, $(\NCL_{4,7})^s=\NCL_{4,12}$),
	\item $\NCL_{4,8}^\beta\skiso \NCL_{4,13}^{\beta}$ (with the above relations, $(\NCL_{4,8}^\beta)^s=\NCL_{4,13}^{\beta}$),
	\item $\NCL_{4,9}\skiso \NCL_{4,14}$ (with the above relations, $(\NCL_{4,9})^s=\NCL_{4,14}$),
	\item $\NCL_{4,10}\skiso \NCL_{4,15}$ (with the above relations, $(\NCL_{4,10})^s=\NCL_{4,15}$),
	\item $\NCL_{4,16}^\alpha\skiso \NCL_{4,17}^{1/\alpha}$,
	\item $\NCL_{4,18}^{\alpha,\beta} \skiso \NCL_{4,18}^{1/\alpha,1/\beta}$,
	\item $\NCL_{4,19}^{\alpha} \skiso \NCL_{4,19}^{1/\alpha}$.
\end{itemize}

Whereas the pairs which are switched copies of one another are trivially skew-isomorphic (that is, with the skew-isomorphism being the identity map), some of the pairs are skew-isomorphic in a nontrivial way, as in the skew-isomorphisms $\NCL_{3,4}^\alpha\skiso \NCL_{3,4}^{1/\alpha}$ and $\NCL_{4,16}^\alpha\skiso \NCL_{4,17}^\alpha$.

\section*{Acknowledgements}

The authors thank the anonymous referees, in particular for suggesting to emphasize the general definition of cohomology for compatible Lie algebras, resulting in a more coherent and streamlined Section~\ref{TheoreticalBases}.

\section*{Declarations}
 The first and second authors are supported by Agencia Estatal de Investigaci\'on (Spain), grant PID2020-115155GB-I00 (European FEDER support included, UE) and by Xunta de Galicia through the Competitive Reference Groups (GRC), ED431C 2023/31.

 The second author is supported by an FCT -- Funda\c c\~ao para a Ci\^encia e a Tecnologia, I.P scholarship with reference number 2023.00796.BD.

 The second and third authors are supported by CMUP, a member of LASI, which is financed by national funds through FCT -- Funda\c c\~ao para a Ci\^encia e a Tecnologia, I.P., under the projects with reference UIDB/00144/2020 and UIDP/00144/2020.

\bibliographystyle{abbrv}

\end{document}